\documentclass{article}

\usepackage{amsmath,amssymb}
\usepackage{amsthm}
\usepackage{latexsym}
\usepackage{indentfirst}
\usepackage{graphicx}
\usepackage{color}
\usepackage{subfigure}
\usepackage{fancyhdr}
\usepackage{helvet} 

\theoremstyle{plain}
\newtheorem{theorem}{Theorem}[section]
\newtheorem{est}[theorem]{Estimate}
\newtheorem{lemma}[theorem]{Lemma}

\theoremstyle{definition}
\newtheorem{defn}[theorem]{Definition}
\newtheorem*{notation}{Notation}

\theoremstyle{remark}
\newtheorem*{remark}{Remark}

\numberwithin{equation}{section}

\DeclareMathOperator{\supp}{supp}

\newcommand{\ud}{\,\mathrm{d}}
\newcommand{\CC}{\mathbb{C}}
\newcommand{\RR}{\mathbb{R}}

\newcommand{\weps}{\widetilde{\epsilon}}

\graphicspath{{figure/}}

\title{Synchrosqueezed Wavelet Transforms: a Tool for Empirical Mode
  Decomposition}
  
\author{Ingrid Daubechies, Jianfeng Lu and Hau-Tieng Wu\\
  Department of Mathematics and \\Program in Applied and Computational
  Mathematics\\ Princeton University, 08544\\
  ingrid@math.princeton.edu, jianfeng@math.princeton.edu,
  hauwu@math.princeton.edu}

\date{}

\begin{document}

\maketitle

\begin{abstract}
  The EMD algorithm, first proposed in \cite{Huang:98}, made more
  robust as well as more versatile in \cite{HuangWuLong:09}, is a
  technique that aims to decompose into their building blocks
  functions that are the superposition of a (reasonably) small number
  of components, well separated in the time-frequency plane, each of
  which can be viewed as approximately harmonic locally, with slowly
  varying amplitudes and frequencies. The EMD has already shown its
  usefulness in a wide range of applications including meteorology,
  structural stability analysis, medical studies -- see, e.g.
  \cite{HuangWu:08}. On the other hand, the EMD algorithm contains
  heuristic and ad-hoc elements that make it hard to analyze
  mathematically.
 
  In this paper we describe a method that captures the flavor and
  philosophy of the EMD approach, albeit using a different approach in
  constructing the components. We introduce a precise mathematical
  definition for a class of functions that can be viewed as a
  superposition of a reasonably small number of approximately harmonic
  components, and we prove that our method does indeed succeed in
  decomposing arbitrary functions in this class. We provide several
  examples, for simulated as well as real data.
\end{abstract}

\section{Introduction}
\label{intro}
Time-frequency representations provide a powerful tool for the
analysis of time dependent signals. They can give insight into the
complex structure of a ``multi-layered'' signal consisting of several
components, such as the different phonemes in a speech utterance, or a
sonar signal and its delayed echo. There exist many types of
time-frequency (TF) analysis algorithms; the overwhelming majority
belong to either ``linear'' or ``quadratic'' methods.

In ``linear'' methods, the signal to be analyzed is characterized by
its inner products with (or correlations with) a pre-assigned family
of templates, generated from one (or a few) basic template by simple
operations. Examples are the windowed Fourier transform, where the
family of templates is generated by translating and modulating a basic
window function, or the wavelet transform, where the templates are
obtained by translating and dilating the basic (or ``mother'')
wavelet. Many linear methods, including the windowed Fourier transform
and the wavelet transform, make it possible to reconstruct the signal
from the inner products with templates; this reconstruction can be for
the whole signal, or for parts of the signal; in the latter case, one
typically restricts the reconstruction procedure to a subset of the TF
plane.  However, in all these methods, the family of template
functions used in the method unavoidably ``colors'' the
representation, and can influence the interpretation given on
``reading'' the TF representation in order to deduce properties of the
signal. Moreover, the Heisenberg uncertainty principle limits the
resolution that can be attained in the TF plane; different trade-offs
can be achieved by the choice of the linear transform or the
generator(s) for the family of templates, but none is ideal, as
illustrated in Figure~\ref{fig_linear_tf_methods}.
\begin{figure}[ht]
\begin{minipage}{.3 \textwidth}
\begin{center}
\includegraphics[ width=.9 \textwidth,height=1.3 in]{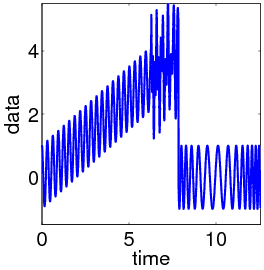}
\end{center}
\end{minipage}
\begin{minipage}{.04 \textwidth}
\begin{center}
\includegraphics[height=1.3 in]{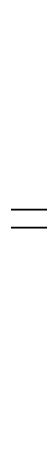}
\end{center}
\end{minipage}
\begin{minipage}{.3 \textwidth}
\begin{center}
\includegraphics[ width=.9 \textwidth, height=1.3 in]{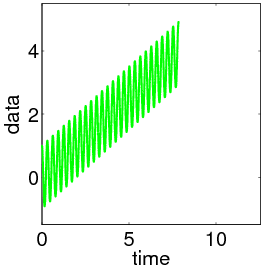}
\end{center}
\end{minipage}
\begin{minipage}{.04 \textwidth}
\begin{center}
\includegraphics[height=1.3 in]{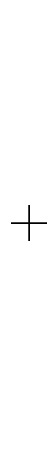}
\end{center}
\end{minipage}
\begin{minipage}{.3 \textwidth}
\begin{center}
\includegraphics[ width=.9 \textwidth, height=1.3 in]{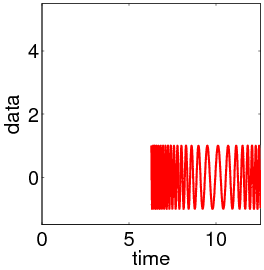}
\end{center}
\end{minipage}

\begin{minipage}{.3 \textwidth}
\vspace*{.9 in}

\begin{center}
\includegraphics[width=.9 \textwidth,  height=1.3 in]{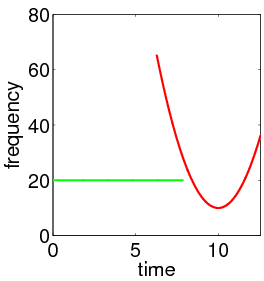}
\vspace*{-.9 in}

\end{center}
\end{minipage}
\hspace*{.04 \textwidth}
\begin{minipage}{.3 \textwidth}
\begin{center}
\includegraphics[width=.9 \textwidth, height=1.3 in]{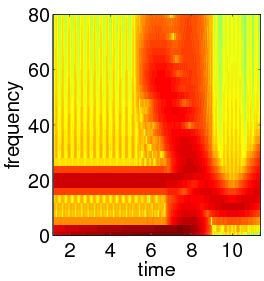}
\end{center}
\end{minipage}
\hspace*{.04 \textwidth}
\begin{minipage}{.3 \textwidth}
\begin{center}
\includegraphics[width=.9 \textwidth, height=1.3 in]{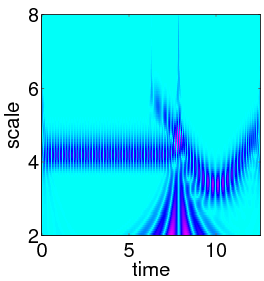}
\end{center}
\end{minipage}

\begin{minipage}{.3 \textwidth}
\vspace*{1.4 in}

$~$
\end{minipage}
\hspace*{.04 \textwidth}
\begin{minipage}{.3 \textwidth}
\begin{center}
\includegraphics[width=.9 \textwidth, height=1.3 in]{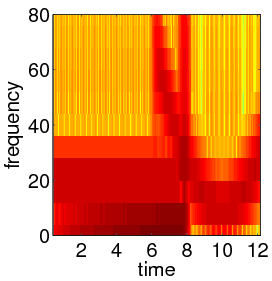}
\end{center}
\end{minipage}
\hspace*{.04 \textwidth}
\begin{minipage}{.3 \textwidth}
\begin{center}
\includegraphics[width=.9 \textwidth, height=1.3 in]{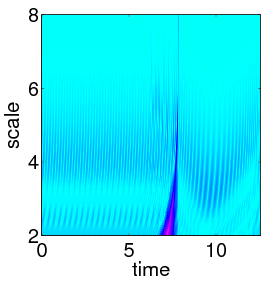}
\end{center}
\end{minipage}
\caption{\label{fig_linear_tf_methods} {\bf Examples of linear
    time-frequency representations.}\hspace*{.35 \textwidth} $~~~$ Top
  row: the signal $s(t)=s_1(t)+s_2(t)$ (top left) defined by $s_1(t) =
  .5t + \cos(20t)$ for $0 \leq t \leq 5 \pi/2$, (top middle) and
  $s_2(t)=\cos \left( \frac{4}{3}\,[ (t-10)^3 - (2 \pi-10)^3] + 10(t-
    2 \pi) \right)$ for $2 \pi \leq t\leq 4 \pi$ (top right); Next
  row: Left:the instantaneous frequency for its two components(left)
  $\omega(t)= 20$ for $0 \leq (t-10)^2 \leq 5 \pi/2$, and $\omega(t)=
  4t^2+10 $ for $2 \pi \leq t \leq 4 \pi$; Middle: two examples of
  (the absolute value of) a continuous windowed Fourier transform of
  $s(t)$, with a wide window (top) and a narrow window (bottom) [these
  are plotted with Matlab, with the `jet' colormap]; Right: two
  examples of a continuous wavelet transform of $s(t)$, with a Morlet
  wavelet (top) and a Haar wavelet (bottom) [plotted with `hsv'
  colormap in Matlab].  The instantaneous frequency profile can be
  clearly recognized in each of these linear TF representations, but
  it is ``blurred'' in each case, in different ways that depend on the
  choice of the transform.  }
\end{figure}
In ``quadratic'' methods to build a TF representation, one can avoid
introducing a family of templates with which the signal is
``compared'' or ``measured''. As a result, some features can have a
crisper, ``more focused'' representation in the TF plane with
quadratic methods (see Figure~\ref{fig_quadratic_tf_methods}).
\begin{figure}[ht]
\begin{center}
\begin{minipage}{.3 \textwidth}
\begin{center}
\includegraphics[ height=1.5 in]{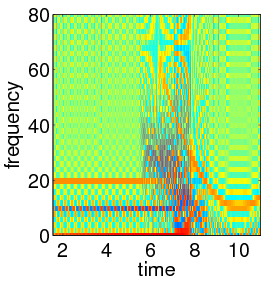}
\end{center}
\end{minipage}
\hspace*{.03 \textwidth}
\begin{minipage}{.3 \textwidth}
\begin{center}
\includegraphics[height=1.5 in]{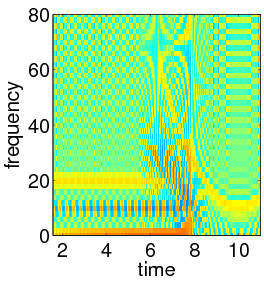}
\end{center}
\end{minipage}
\hspace*{.03 \textwidth}
\begin{minipage}{.3 \textwidth}
\begin{center}
\includegraphics[ height=1.5 in]{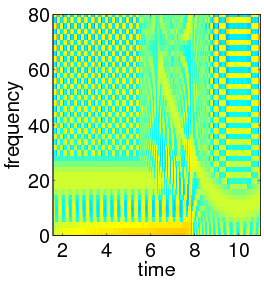}
\end{center}
\end{minipage}
\caption{\label{fig_quadratic_tf_methods} {\bf Examples of quadratic
    time-frequency representations.}  Left: the Wigner-Ville transform
  of $s(t)$, in which interference causes the typical Moir\'{e}
  patterns; Middle and Right: two pseudo-Wigner-Ville transforms of
  $s(t)$, obtained by blurring the Wigner-Ville transform slightly
  (middle) and somehwat more (right). [All three graphs plotted with
  `jet' colormap in Matlab, calibrated identically.] The blurring
  removes the interference patterns, at the cost of precise location
  in the time-frequency localization.  }
\end{center}
\end{figure}
However, in this case, ``reading'' the TF representation of a
multi-component signal is rendered more complicated by the presence of
interference terms between the TF representations of the individual
components; these interference effects also cause the ``time-frequency
density'' to be negative in some parts of the TF plane.  These
negative parts can be removed by some further processing of the
representation \cite{Flandrin:99}, at the cost of reintroducing some blur
in the TF plane again.  Reconstruction of the signal, or part of the
signal, is much less straightforward for quadratic than for linear TF
representations.

In many practical applications, in a wide range of fields (including,
e.g., medicine and engineering) one is faced with signals that have
several components, all reasonably well localized in TF space, at
different locations. The components are often also called
``non-stationary'', in the sense that they can present jumps or
changes in behavior, which it may be important to capture as
accurately as possible. For such signals both the linear and quadratic
methods come up short. Quadratic methods obscure the TF representation
with interference terms; even if these could be dealt with,
reconstruction of the individual components would still be an
additional problem. Linear methods are too rigid, or provide too
blurred a picture.  Figures~\ref{fig_linear_tf_methods},
\ref{fig_quadratic_tf_methods} show the artifacts that can arise in
linear or quadratic TF representations when one of the components
suddenly stops or starts.  Figure~\ref{fig_exs_am_fm} shows examples
of components that are non harmonic, but otherwise perfectly
reasonable as candidates for a single-component-signal, yet not well
represented by standard TF methods, as illustrated by the lack of
concentration in the time-frequency plane of the transforms of these
signals.
\begin{figure}[ht]
\begin{center}
\begin{minipage}{.22 \textwidth}
\begin{center}
\includegraphics[width=.9 \textwidth,height=1.3 in]{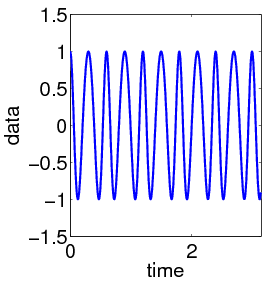}
\end{center}
\end{minipage}
\hspace*{.02 \textwidth}
\begin{minipage}{.22 \textwidth}
\begin{center}
\includegraphics[width=.9 \textwidth,height=1.3 in]{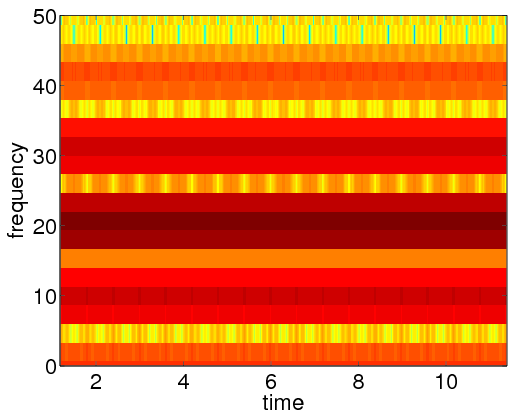}
\end{center}
\end{minipage}
\hspace*{.02 \textwidth}
\begin{minipage}{.22 \textwidth}
\begin{center}
\includegraphics[width=.9 \textwidth,height=1.3 in]{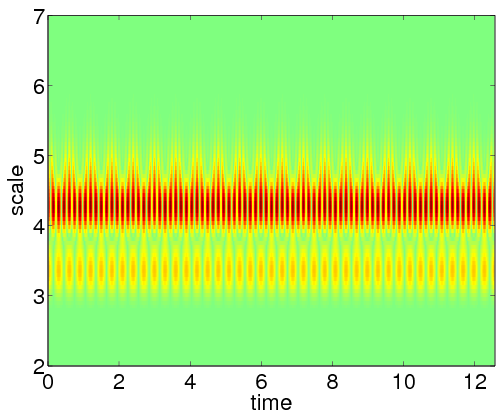}
\end{center}
\end{minipage}
\hspace*{.02 \textwidth}
\begin{minipage}{.22 \textwidth}
\begin{center}
\includegraphics[ width=.9 \textwidth,height=1.3 in]{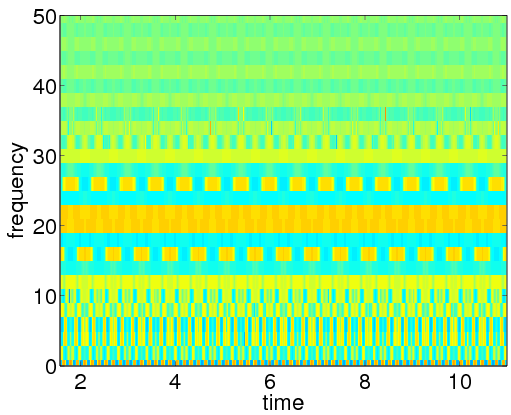}
\end{center}
\end{minipage}
\end{center}

\begin{center}
\begin{minipage}{.22 \textwidth}
\begin{center}
\includegraphics[ width=.9 \textwidth,height=1.3 in]{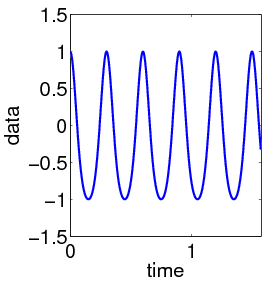}
\end{center}
\end{minipage}
\hspace*{.02 \textwidth}
\begin{minipage}{.22 \textwidth}
\begin{center}
\includegraphics[ width=.9 \textwidth,height=1.3 in]{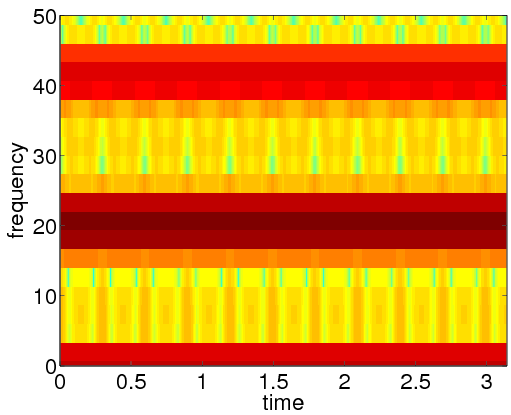}
\end{center}
\end{minipage}
\hspace*{.02 \textwidth}
\begin{minipage}{.22 \textwidth}
\begin{center}
\includegraphics[width=.9 \textwidth,height=1.3 in]{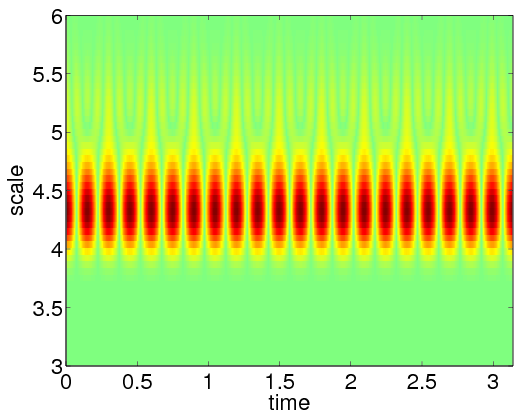}
\end{center}
\end{minipage}
\hspace*{.02 \textwidth}
\begin{minipage}{.22 \textwidth}
\begin{center}
\includegraphics[ width=.9 \textwidth,height=1.3 in]{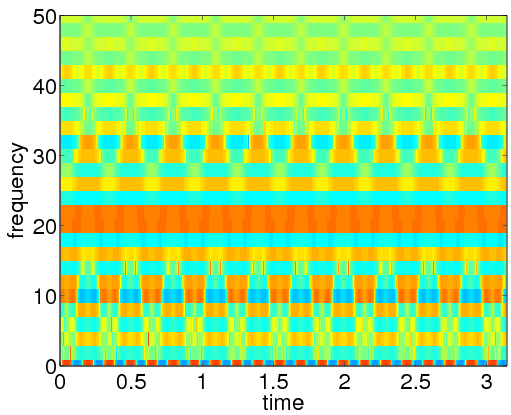}
\end{center}
\end{minipage}
\end{center}
\caption{\label{fig_exs_am_fm}Two examples of wave functions of the
  type $\cos[\phi(t)]$, with slowly varying $A(t)$ and $\phi'(t)$, for
  which standard TF representations are not very well localized in the
  time-frequence plane. Left: signal, Middle left: windowed Fourier
  transform; Middle Right: Morlet wavelet transform; Right:
  Wigner-Ville function. [All TF representations plotted with `jet'
  colormap in Matlab.]  }
\end{figure}
The {\em Empirical Mode Decomposition} (EMD) method was proposed by
Norden Huang \cite{Huang:98} as an algorithm that would allow
time-frequency analysis of such multicomponent signals, without the
weaknesses sketched above, overcoming in particular artificial
spectrum spread caused by sudden changes.  Given a signal $s(t)$, the
method decomposes it into several {\em instrinsic mode functions}
(IMF):
\begin{equation}\label{eq:EMD}
  s(t) \,=\, \sum_{k=1}^K s_k(t),
\end{equation}
where each IMF is basically a function oscillating around 0, albeit
not necessarily with constant frequency:
\begin{equation}\label{eq:IMF}
  s_k(t) \,=\, A_k(t)\, \cos(\phi_k(t))~,~\mbox{ with } A_k(t),\,\phi'_k(t)>0 ~ \forall t~.
\end{equation}
Essentially, each IMF is an amplitude modulated-frequency modulated
(AM-FM) signal; typically, the change in time of $A_k(t),\,\phi'_k(t)$
is much slower than the change of $\phi_k(t)$ itself, which means that
locally (i.e. in a time interval $[t-\delta,t+\delta]$, with $\delta
\approx 2\pi [\phi'_k(t)]^{-1}$) the component $s_k(t)$ can be
regarded as a harmonic signal with amplitude $A_k(t)$ and frequency
$\phi_k'(t)$.  (In \cite{Huang:98}, the conditions on an IMF are
phrased as follows: (1) in the whole data set, the number of extrema
and the number of zero crossings of $s_k(t)$ must either be equal or
differ at most by one; and (2) at any $t$, the value of a smooth
envelope defined by the local minima of the IMF is the negative of the
corresponding envelope defined by the local maxima.)  After the
decomposition of $s(t)$ into its IMF components, the EMD algorithm
proceeds to the computation of the ``instantaneous frequency'' of each
component. Theoretically, this is given by
$\omega_k(t)\,:=\,\phi'_k(t)$; in practice, rather than a (very
unstable) differentiation of the estimated $\phi_k(t)$, the originally
proposed EMD method used the Hilbert transform of the $s_k(t)$
\cite{Huang:98}; more recently, this has been replaced by other
methods \cite{HuangWuLong:09}.

It is obvious that every function can be written in the form
\eqref{eq:EMD} with each component as in \eqref{eq:IMF}. If $s(t)$ is
supported (or observed) in $[-T,T]$, then the Fourier series on
$[-T,T]$ of $s(t)$ is actually such a decomposition. It is also easy
to see that such a decomposition is far from unique. This is simply
illustrated by considering the following signal:
\begin{equation}\label{eq:threecosine}
s(t)\, =\,.25\, \cos([\Omega-\gamma] t)\, +\,2.5\, \cos (\Omega t)\,+\,.25\,\cos([\Omega+\gamma] t)
\,= \,\left(\,2\,+\,\cos^2\left[\,\frac{\gamma}{2}\,t\,\right]\,\right)\,\cos(\, \Omega t)~,
\end{equation}
where $\Omega \gg \gamma$, so that one can set
$A(t)\,:=\,2\,+\,\cos^2\left[\,\frac{\gamma}{2}\,t\,\right]\,,$ which varies 
much more slowly than $\cos[\phi(t)]\,=\,\cos[\Omega \,t]$. 
\begin{figure}[ht]
\begin{center}
\includegraphics[height=1.5 in]{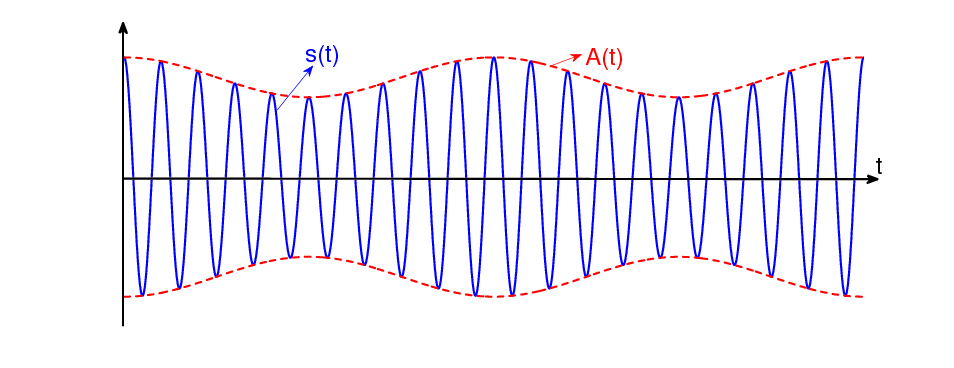}
\caption{\label{fig_3_cosines}{\b Non-uniqueness for decomposition
    into IMT.}  This function can be considered as a single component,
  of the type $A(t)\cos[\Omega t]$, with slowly varying amplitude, or
  as the sum of three components. (See text.)  }
\end{center}
\end{figure}
The interpretation of this signal is not unique: It can be regarded
either as a summation of three cosines with frequencies $\Omega
-\gamma$, $\Omega$ and $\Omega+ \gamma$ respectively, or as a single
component with frequency $\Omega$ which has an amplitude $A(t)$ that
is slowly modulated.  Depending on the circumstances, either
interpretation can be the ``best''.  In the EMD's framework, the
second interpretation (single component, with slowly varying
amplitude) is preferred when $\Omega \gg \gamma$; the EMD is typically
applied when it is more ``physically meaningful'' to decompose a
signal into {\em fewer} components if this can be achieved by mild
variations in frequency and amplitude; in those circumstances, this
preference is sensible.  The (toy) example illustrates that we should
not expect a universal solution to all TF decomposition problems. For
certain classes of functions, consisting of a (reasonably) small
number of components, well separated in the TF plane, each of which
can be viewed as approximately harmonic locally, with slowly varying
amplitdes and frequencies, it is clear, however, that a technique that
identifies these components accurately, even in the presence of noise,
has great potential for a wide range of applications.  Such a
decomposition should be able to accommodate such mild variations
within the building blocks of the decomposition.

The EMD algorithm, first proposed in \cite{Huang:98}, made more robust
as well as more versatile in \cite{HuangWuLong:09} (an extension to
higher dimensions is now possible), is such a technique.  It has
already shown its usefulness in a wide range of applications including
meteorology, structural stability analysis, medical studies -- see,
e.g. \cite{Costa:07, Cummings:04, Huang:98}; a recent review is given
in \cite{HuangWu:08}.  On the other hand, the EMD algorithm contains a
number of heuristic and ad-hoc elements that make it hard to analyze
mathematically its guarantees of accuracy or the limitations of its
applicability.  For instance, the EMD algorithm, uses a {\em sifting
  process} to construct the decomposition of type \eqref{eq:EMD}. In
each step in this sifting process, two smooth interpolating functions
are constructed (using cubic splines), one of the local maxima
($\bar{s}(t)$), and one of the local minima ($\underline{s}(t)$). From
these interpolates, a mean curve of the signal is defined as $m(t) =
(\bar{s}(t) + \underline{s}(t))/2$, which is then subtracted from the
signal: $r_1(t) = s(t) - m(t)$. In most cases, $r_1$ is not yet a
satisfactory IMF; the process is then repeated on $r_1$ again, etc
$\ldots$; this repeated process is called ``sifting''.  Sifting is
done for either a fixed number of times, or until a certain stopping
criterium is satisfied; the final remainder $r_n(t)$ is taken as the
first IMF, $s_1:=r_n$.  The algorithm continues with the difference
between the original signal and the first IMF to extract the second
IMF (which is the first IMF obtained from the ``new starting signal''
$s(t)-s_1(t)$) and so on. (Examples of the decomposition will be given
in Section~\ref{numerical}.)  Because the sifting process relies
heavly on interpolates of maxima and minima, the end result has some
stability problems in the presence of noise, as illustrated in
\cite{WuHuang:09}. The solution proposed in \cite{WuHuang:09}
addresses these issues in practice, but poses new challenges to our
mathematical understanding.

Attempts at a mathematical understanding of the approach and the
results produced by the EMD method have been mostly exploratory.  A
systematic investigation of the performance of EMD acting on white
noise was carried out in \cite{FlandrinRillingGoncalves:04,
  WuHuang:04}; it suggests that in some limit, EMD on signals that
don't have structure (like white noise) produces a result akin to
wavelet analysis. The decomposition of signals that are superpositions
of a few cosines was studied in \cite{RillingFlandrin:08}, with
interesting results.  A first different type of study, more aimed at
building a mathematical framework, is given in \cite{LinWangZhou:09,
  HuangYangWang:09}, which analyzes mathematically the limit of an
infinite number of ``sifting'' operations, showing it defines a
bounded operator on $\ell_{\infty}$, and studies its mathematical
properties.

In summary, the EMD algorithm has shown its usefulness in various
applications, yet our mathematical understanding of it is still very
sketchy. In this paper we discuss a method that captures the flavor
and philosophy of the EMD approach, without necessarily using the same
approach in constructing the components.  We hope this approach will
provide new light in understanding of what makes EMD work, when it can
be expected to work (and when not) and what type of precision we can
expect.

\section{Synchrosqueezing Wavelet Transforms}
\label{synchro}

Synchrosqueezing was introduced in the context of analyzing auditory
signals \cite{DaubechiesMaes:96}; it is a special case of {\em
  reallocation methods} \cite{AugerFlandrin:95,
  Chassande-MottinAugerFlandrin:03,
  Chassande-MottinDaubechiesAuger:97}, which aim to ``sharpen'' a
time-frequency representation $\mathcal{R}(t,\omega)$ by
``allocating'' its value to a different point $(t',\omega')$ in the
time-frequency plane, determined by the local behavior of
$\mathcal{R}(t,\omega)$ around $(t,\omega)$.  In the case of
synchrosqueezing, one reallocates the coefficients resulting from a
continuous wavelet transform to get a concentrated time-frequency
picture, from which instantaneous frequency lines can be extracted.

To motivate the idea, let us start with a purely harmonic signal,
\[
s(t) = A \cos(\omega t).
\]
Take a wavelet $\psi$ that is concentrated on the positive frequency axis:
$\hat{\psi}(\xi)=0$ for $\xi<0$. Denote by $W_s(a,b)$ the continuous wavelet
transform of $s$ defined by this choice of $\psi$. We have
\begin{equation}
\begin{split}
  W_s(a,b) & = \int s(t) \,a^{-1/2}\, \overline{\psi\bigl(\frac{t-b}{a}\bigr)} \ud t\\
  & = \frac{1}{2\pi} \int \hat{s}(\xi) \,a^{1/2} \,\overline{\hat{\psi}(a\xi)}\,e^{ib\xi} \ud \xi\\
  & = \frac{A}{4\pi} \int [\delta(\xi-\omega) + \delta(\xi+\omega)]\,
  a^{1/2}
  \,\overline{\hat{\psi}(a\xi)}\,e^{ib\xi} \ud \xi\\
  & = \frac{A}{4\pi} \,a^{1/2}\,
  \overline{\hat{\psi}(a\omega)}\,e^{ib\omega}.
\end{split}
\end{equation}
If $\hat{\psi}(\xi)$ is concentrated around $\xi=\omega_0$, then $W_s(a,b)$
will be concentrated around $a = \omega_0/\omega$. However, the wavelet transform $W_s(a,b)$ will be spread
out over a region around the horizontal line $a = \omega_0/\omega$ on the time-scale plane.
The observation made in \cite{DaubechiesMaes:96} is that although $W_s(a,b)$ is spread out in $a$, its oscillatory behavior in $b$ points to the original frequency $\omega$, regardless of the value of $a$.   

This led to the suggestion to compute, for any $(a,b)$ for which $W_s(a,b)\neq 0$,
a candidate instantaneous frequency $\omega(a,b)$ by
\begin{equation}
\label{eq:omega}
\omega(a,b) = -i (W_s(a,b))^{-1} \frac{\partial}{\partial b} W_s(a,b).
\end{equation}
For the purely harmonic signal $s(t) = A \cos(\omega t)$, one obtains 
$\omega(a,b) = \omega$, as desired; this is illustrated in Figure \ref{fig:harmonic}
\begin{figure}[ht]
\centering
\begin{minipage}{.3 \textwidth}
\begin{center}
\includegraphics[width=0.9\textwidth, height=1.3 in]{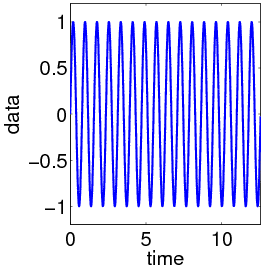}
\end{center}
\end{minipage}
\hspace*{.03 \textwidth}
\begin{minipage}{.3 \textwidth}
\begin{center}
\includegraphics[width=0.9\textwidth, height=1.3 in]{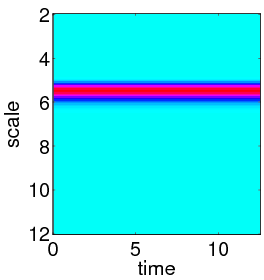}
\end{center}
\end{minipage}
\hspace*{.03 \textwidth}
\begin{minipage}{.3 \textwidth}
\begin{center}
\includegraphics[width=0.9\textwidth, height=1.3 in]{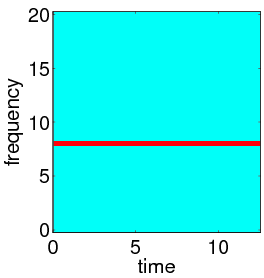}
\end{center}
\end{minipage}
\caption{ Left: the harmonic signal $ f(t)=\sin(8 t)$; Middle: the
  continuous wavelet transform of $f$; Right: synchrosqueezed
  transform of $f$. \label{fig:harmonic}}
\end{figure}
In a next step, the information from the time-scale plane is
transferred to the time-frequency plane, according to the map $(b,a)
\longrightarrow (b,\omega(a,b))$, in an operation dubbed {\em
  synchrosqueezing}. In \cite{DaubechiesMaes:96}, the frequency
variable $\omega$ and the scale variable $a$ were ``binned'', i.e.
$W_s(a,b)$ was computed only at discrete values $a_k$, with
$a_k-a_{k-1}=(\Delta a)_k$, and its synchrosqueezed transform
$T_s(\omega,b)$ was likewise determined only at the centers
$\omega_{\ell}$ of the successive bins
$\left[\omega_{\ell}-\frac{1}{2}\Delta
  \omega,\omega_{\ell}+\frac{1}{2}\Delta \omega\right]$, with
$\omega_{\ell}-\omega_{\ell-1}=\Delta \omega$, by summing different
contributions:
\begin{equation}\label{eq:squeeze}
T_s(\omega_{\ell},b) = (\Delta \omega)^{-1}\sum_{a_k: \lvert \omega(a_k,b)- \omega_l\rvert\leq \Delta\omega/2} W_s(a_k,b)\,a_k^{-3/2}\,(\Delta a)_k.
\end{equation}
The following argument shows that the signal can still be reconstructed after the synchrosqueezing.
We have
\begin{equation}
\begin{split}
\int_0^{\infty} W_s(a,b) \,a^{-3/2} \ud a
& = \frac{1}{2\pi} \int_{-\infty}^{\infty}\int_0^{\infty} \hat{s}(\xi)\, \overline{\hat{\psi}(a\xi)}\, e^{ib\xi}\,a^{-1} \ud a\ud\xi \\
& = \frac{1}{2\pi} \int_0^{\infty}\int_0^{\infty} \hat{s}(\xi) \,\overline{\hat{\psi}(a\xi)}\, e^{ib\xi}\,a^{-1} \ud a\ud\xi \\
& = \int_0^{\infty} \overline{\hat{\psi}(\xi)}\,\frac{\ud\xi}{\xi}\cdot
\frac{1}{2\pi} \int_0^{\infty} \hat{s}(\zeta)\, e^{ib\zeta} \ud\zeta .
\end{split}
\end{equation}
Setting $C_{\psi}\,=\,2\int_0^{\infty} \overline{\hat{\psi}(\xi)}\,\frac{\ud\xi}{\xi}$, we then obtain (assuming that
$s$ is real, so that $\hat{s}(\xi)=\overline{\hat{s}(-\xi)}$, hence $s(b) = (4\pi)^{-1} \mathfrak{Re}\left[\int_0^{\infty} \hat{s}(\xi)\, e^{ib\xi} \ud\xi\,\right]\,$)
\begin{equation}
s(b) = \mathfrak{Re}\left[C_{\psi}^{-1}\, \int_0^{\infty} W_s(a,b) \,a^{-3/2} \ud a\,\right].
\end{equation}
In the piecewise constant approximation corresponding to the binning in $a$, this becomes
\begin{equation}
s(b) \approx \mathfrak{Re}\left[C_{\psi}^{-1}\, \sum_k W_s(a_k,b) \,a_k^{-3/2} \,(\Delta a)_k \,\right]\,=\, \mathfrak{Re}\left[C_{\psi}^{-1}\, \sum_{\ell} T_s(\omega_{\ell},b)\,(\Delta \omega)\right].
\end{equation}

\begin{remark}
As defined above, \eqref{eq:squeeze} implicitly assumes a linear scale discretization
of $\omega$. If instead logarithmic discretization is used, the $\Delta \omega$ has to be made dependent on $\ell$; alternatively, one can also change 
the exponent of $a$ from $-3/2$ to $-1/2$.
\end{remark}

If one chooses (as we shall do here) to continue to treat $a$ and $\omega$ as continuous variables, without discretization, the analog of \eqref{eq:squeeze} is
\begin{equation}
\label{eq:squeeze_cont}
\mathcal{T}_s(\omega,b) =  \int_{A(b)} W_s(a,b) \, a^{-3/2} \, \delta(\,\omega(a,b)-\omega\,) \ud a ,
\end{equation}
where $A(b)=\{\,a\,;\,W_s(a,b) \neq 0 \,\}$, and $\omega(a,b)$ is as defined in 
\eqref{eq:omega} above, for $(a,b)$ such that $a \in \,A(b)$.

\begin{remark}
In practice, the determination of those $(a,b)$-pairs for which $W_s(a,b)=0$ is rather unstable, when $s$ has been contaminated by noise. For this reason, it is often useful
to consider a threshold for $\lvert W_s(a,b) \rvert$, below which $\omega(a,b)$ is not defined; this amounts to replacing $A(b)$ by the smaller region $A_{\epsilon}(b):=\{a\,;\,\lvert W_s(a,b) \rvert\geq \epsilon\,\}$.  
\end{remark}

One can also view
synchrosqueezing as follows. For sufficiently ``nice'' $s$ and $\psi$, we have, for $a \in A(b)$, 
\begin{align}
\omega(a,b) &= -i (W_s(a,b))^{-1} \frac{\partial}{\partial b} W_s(a,b) \\
&= \frac{\int \xi \,\hat{s}(\xi)\,\overline{\hat{\psi}(a\xi)}\, e^{ib\xi} \ud\xi}
{\int \hat{s}(\xi)\,\overline{\hat{\psi}(a\xi)}\, e^{ib\xi}\ud\xi} \\
&= \frac{\int -i s'(t)\,a^{-1/2}\,\overline{\psi\bigl(\frac{t-b}{a}\bigr)}\ud t}
{\int s(t)\,a^{-1/2}\,\overline{\psi\bigl(\frac{t-b}{a}\bigr)}\ud t} \\
&= \frac{\int -i \,(s+iHs)'(t)\,a^{-1/2}\,\overline{\psi\bigl(\frac{t-b}{a}\bigr)}
\ud t}{\int (s+iHs)(t)\,a^{-1/2}\,\overline{\psi\bigl(\frac{t-b}{a}\bigr)}\ud t},
\end{align}
where $Hs$ denotes the Hilbert transform of $s$; the last equality uses
that $\hat{\psi}$ is supported on the positive frequencies only.
If now $s$ is a so-called ``asymptotic signal'', i.e., if
\begin{equation}
s(t) = a(t)\cos(\phi(t)),
\end{equation}
with $a'(t)\ll 1$ and $\phi''(t) \ll \phi'(t)$, then the Hilbert
transform of $s$ is (approximately) given
by
\begin{equation}
Hs(t) \sim a(t)\sin(\phi(t)).
\end{equation}
Therefore, using the above expression of
$\omega(a,b)$, we have approximately
\begin{align}
\omega(a,b) &\sim
\frac{\int a(t)\,\phi'(t)\,e^{i\phi(t)}\,a^{-1/2}\,\overline{\psi\bigl(\frac{t-b}{a}
\bigr)}\ud t}{\int a(t)\,e^{i\phi(t)}\,a^{-1/2}\,\overline{\psi\bigl(\frac{t-b}{a}
\bigr)}\ud t}\sim \phi'(b),
\end{align}
where, in the first approximation, we have omitted the term containing $a'(t)$ as $a'(t)
\ll \phi'(t)$, and in the second approximation, we have used that $\psi$ is
localized around $0$.

This heuristic argument suggests that, for asymptotic signals, synchrosqueezing an 
appropriate wavelet transform will 
indeed give a single
line on the time-frequency plane, at the value of the 
``instantaneous frequency'' of a (putative) IMF.

\section{Main Result}
\label{main_result}
We define a class of functions, containing {\em intrinsic mode type}
components that are {\em well-separated}, and show that they can be
identified and characterized by means of synchrosqueezing.

We start with the following definitions:

\begin{defn}{\bf{[Intrinsic Mode Type Function]}} \\
A function 
$f: \RR \rightarrow \CC $ is said to be {\em intrinsic-mode-type (IMT) with
accuracy $\epsilon\,>\,0$} if $f$ and $A\,:=\,|f|$ have the following properties:
\begin{eqnarray}
f(t)\, =&\, A(t)\, e^{i \phi(t)} ~~~~~ \mbox{ where }& A \in C^1(\RR), \, \phi \in C^2(\RR)\nonumber\\
&& \inf_{t\in\RR}\phi'(t)  \,>\, 0 ~~,\nonumber\\
&& |A'(t)|,\, |\phi''(t)| \leq \epsilon \,|\phi'(t)|\,, \,\forall t \in \RR\nonumber\\  
&& M''\,:=\,\sup_{t \in \RR}|\phi''(t)|< \infty~.\nonumber
\end{eqnarray}
\end{defn}

\begin{defn}{\bf{[Superposition of Well-Separated Intrinsic Mode Components]}}\\
A function 
$f: \RR \rightarrow \CC $ is said to be a superposition of, or to consist of,
{\em well-separated Intrinsic Mode Components, up to accuracy $\epsilon$, 
and with separation $d$} if it can be written as 
\[
f(t)\,=\,\sum_{k=1}^K\, f_k(t)\, 
\]
where all the $f_k$ are IMT, and where moreover their respective phase functions $\phi_k$ satisfy
\[
\phi'_{k}(t)>\phi'_{k-1}(t)\,,~~ \mbox{ and } ~~~
 |\phi'_k(t) \,-\, \phi'_{k-1}(t)| \geq d [\phi'_k(t) \,+\, \phi'_{k-1}(t)]\,,\quad\forall t \in \RR\,.
\]
\end{defn}

\begin{remark}
It is not really necessary for the components $f_k$ to be defined on all of $\RR$. One can also suppose that they are supported on intervals,
$\supp(f_k)\,=\,\supp(A_k) \subset [-T_k,T_k]$, where the different $T_k$
need not be identical. In this case the various inequalities governing the
definition of an IMT function or a superposition of well-separated IMT components must simply be restricted to the relevant intervals. For the 
inequality above on the $\phi'_{k}(t), \,\phi'_{k-1}(t)$, it may happen that some $t$ are covered by (say) $[-T_k,T_k]$ but not by 
$[-T_{k-1},T_{k-1}]$; one should then replace $k-1$ by the largest 
$\ell <k$ for which $t \in [-T_{\ell},T_{\ell}]$;
other, similar, changes would have to be made if $t \in [-T_{k-1},T_{k-1}]\setminus [-T_k,T_k]$.\\
We omit this extra wrinkle for the sake of keeping notations manageable.
\end{remark}

\begin{notation}{\bf{[Class $\mathcal{A}_{\epsilon, d}$]}}\\
We denote by $\mathcal{A}_{\epsilon, d}$ the set of all superpositions of well-separated IMT, up to accuracy $\epsilon$ and 
with separation $d$.
\end{notation}

Our main result is then the following:

\begin{theorem}{\bf{[Main result]}}\\
Let $f$ be a function in $\mathcal{A}_{\epsilon,d}$, and set $\widetilde{\epsilon}:=\epsilon^{1/3}$. Pick a wavelet
$\psi$ such that its Fourier transform $\widehat{\psi}$ is supported 
in $[1-\Delta,1+\Delta]$, with $\Delta <d/(1+d)$, and set
$\mathcal{R}_{\psi}\,=\,\sqrt{2\pi}\,\int\,\widehat{\psi}(\zeta) \, \zeta^{-1}\,d\zeta\,$.
Consider the continuous wavelet transform $W_f(a,b)$ of $f$ with respect to this wavelet, 
as well as the function $S_{f,\sigma}(b,\omega)$ obtained by
synchrosqueezing $W_f$, with threshold $\weps$, i.e.
\[
S_{f,\weps}(b,\omega)\,:=\,\int_{A_{\weps,f}(b)}\,W_f(a,b)\,\delta(\omega -\omega_{f}(a,b))\,a^{-3/2}\,da~,
\]
where $A_{\weps,f}(b)\,:=\,\{a \in \RR_+\,;\,|W_f(a,b)| \,>\,\weps \,\}~.$\\
Then, provided $\epsilon$ (and thus also $\weps$) is sufficiently small, the following hold:\\
~\\
$\bullet \, |W_f(a,b)|\,>\,\weps$ only when, for some 
$k \in \{1,\ldots,K\}~ $, $(a,b) \in Z_k\,:=\,$
$\{(a,b)\,;\, |\,a\,\!\phi'_k(b)\,-\,1\,|\,<\,\Delta\}\,$. \\
~\\
$\bullet$ For each $k \in \{1,\ldots,K\}$, and
for each pair $(a,b) \in Z_k$ for which
$|W_f(a,b)|\,>\,\weps$, we have
\[
|\omega_f(a,b)\,-\, \phi'_k(b)\,|\,\leq\,\weps~.
\]
$\bullet$ Moreover, for each $k \in \{1,\ldots,K\}$, there exists a
constant $C$ such that, for any $b \in \RR$,
\[
\left|\,\left(\,\mathcal{R}_{\psi}^{-1} \int_{|\omega-\phi'_k(b)|<\weps}\,S_{f,\weps}(b,\omega)\,d\omega\,\right)\,-\,A_k(b)\,e^{i\,\phi_k(b)}\,\right|\,\leq\, C \,\weps~.\nonumber
\]
\label{mainthm}
\end{theorem}

This theorem basically tells us that, for $f \in \mathcal{A}_{\epsilon, d}$, 
the synchrosqueezed version $S_{f,\weps}$ of the wavelet transform
$W_f$ is completely concentrated, in the $(t,\omega)$-plane, in narrow
bands around the curves $\omega\,=\,\phi'_k(t)$, and that the
restriction of $S_{f,\weps}$ to the $k$-th narrow band suffices to
reconstruct, with high precision, the $k$-th IMT component of $f$.
Synchrosqueezing (an appropriate) wavelet transform thus provides the
{\em adaptive time-frequency decomposition} that is the goal of
Empirical Mode Decomposition.

The proof of Theorem \ref{mainthm} relies on a number of estimates,
which we demonstrate one by one, at the same time providing more
details about what it means for $\epsilon$ to be ``sufficiently
small''.  In the statement and proof of all the estimates in this
section, we shall always assume that all the conditions of Theorem
\ref{mainthm} are satisfied (without repeating them), unless stated
otherwise.

The first estimate bounds the growth of the $A_k$, $\phi'_k$ in the neighborhood of $t$, in terms of the value of $|\phi'_k(t)|$.
\begin{est}
For each $k \in \{1,\ldots,K\}$, we have
\[
|A_k(t+s)\,-\,A_k(t)| \,\leq \, \epsilon\,|s|\,\left(\,|\phi'_k(t)|\,+\,\frac{1}{2}\,M''_k\,|s| \,\right)
~~ \mbox{ and } ~~ |\phi'_k(t+s)\,-\,\phi'_k(t)| \,\leq \,\epsilon\,|s|\,\left(\,|\phi'_k(t)|\,+\,\frac{1}{2}\,M''_k\,|s| \,\right) ~.
\]
\label{est1}
\end{est}

\begin{proof}
\begin{eqnarray}
\left|\,A_k(t+s)\,-\,A_k(t)\,\right|\,&=& \,\left|\, \int_0^s\, A'_k(t+u)\,du\,\right|\nonumber \\
& \leq & \, \int_0^s\,\left|\, A'_k(t+u)\right|\, du \, \leq\, \epsilon \, \int_0^s\,\left|\, \phi'_k(t+u)\right|\, du~. \nonumber \\
&= & \,\epsilon \, \int_0^s\,\left|\, \phi'_k(t)\,+\,\int_0^u\,\phi''_k(t+u)\,du\,\right|\,dt
\,\leq\,\epsilon \,\left(\,\left|\phi'_k(t)\right|\,|s|\,+\,\frac{1}{2}\,M''_k\,|s|^2\,\right)~.\nonumber
\end{eqnarray}
The other bound is analogous.
\end{proof}

The next estimate shows that, for $f$and $\psi$ as given in the statement of Theorem \ref{mainthm}, the
wavelet transform $W_f(a,b)$ is concentrated near the regions where, for some $k \in \{1,2,\ldots,K\}$, $~a\,\!\phi'_k(b)~$ is close to 1. 

\begin{est}
\[
\left|\,W_f(a,b)\,-\,\sqrt{2 \pi}\,\sum_{k=1}^K\,A_k(b)\,e^{i\phi_k(b)}\,\sqrt{a}\,\widehat{\psi}\left(a\,\phi'_k(b)\right)\,\right|\,\leq \,\epsilon\,a^{3/2}\,\Gamma_1 ~,
\]
where
\[
\Gamma_1\,:=\, I_1 \sum_{k=1}^K\,|\phi_k'(b)|
\,+\, \frac{1}{2}\,I_2\,a\,\sum_{k=1}^K\,\left[\,M''_k\,+\, |A_k(b)|\,|\phi_k'(b)|\, \right]
\,+\, \frac{1}{6}\,I_3\,a^2\,\sum_{k=1}^K\,M''_k\,|A_k(b)|~,
\]
with $I_n\,:=\,\int\,|u|^n \, |\psi(u)|\,du~$.
\label{est2}
\end{est}

\begin{proof}
We have 
\begin{eqnarray}
W_f(a,b)\,&=&\,\sum_{k=1}^K\, \int \, A_k(t) \, e^{i\phi_k(t)}\, 
a^{-1/2} \, \psi\left(\frac{t-b}{a}\right)\, dt\nonumber\\
&=&\, \sum_{k=1}^K\,A_k(b)\,\int\,e^{i[\phi_k(b)
\,+\,\phi_k'(b)\,(t-b)\,+\,\int_0^{t-b}[\phi_k'(b+u)-\phi_k'(b)]du]}\, 
a^{-1/2} \, \psi\left(\frac{t-b}{a}\right)\, dt\nonumber\\
&&~~~~~~~~~\,+\,\sum_{k=1}^K\,[A_k(t)\,-\,A_k(b)]\, e^{i\phi_k(t)}\, 
a^{-1/2} \, \psi\left(\frac{t-b}{a}\right)\, dt~.
\nonumber
\end{eqnarray}
It follows that
\begin{eqnarray}
&&\left|\,W_f(a,b)\,-\,\sqrt{2 \pi}\,\sum_{k=1}^K\,A_k(b)\,e^{i\phi_k(b)}\,\sqrt{a}\,\widehat{\psi}\left(a\,\phi'_k(b)\right)\,\right|\,\nonumber\\
&&\quad\leq \,\sum_{k=1}^K\,\int\,\left| A_k(t)\,-\,A_k(b)\right|\,a^{-1/2} \, \left|\psi\left(\frac{t-b}{a}\right)\right|\, dt\nonumber\\
&&\quad\quad\quad\quad\,+\, \sum_{k=1}^K\,|A_k(b)|\,\int\,\left|e^{i \int_0^{t-b}[\phi'_k(b+u)-\phi'_k(b)]du}\,-\,1\right|\,a^{-1/2} \, \left|\psi\left(\frac{t-b}{a}\right)\right|\, dt\nonumber\\
&&\quad\leq \,\sum_{k=1}^K\,\int\,\epsilon\,|t-b|\left(|\phi'_k(b)|+\frac{1}{2}\,M_k''\,|t-b|\right)\,a^{-1/2} \, \left|\psi\left(\frac{t-b}{a}\right)\right|\, dt\nonumber\\
&&\quad\quad\quad\quad\,+\, \sum_{k=1}^K\,|A_k(b)|\,\int\,\left|\int_0^{t-b}[\phi'_k(b+u)-\phi'_k(b)]du\right|\,a^{-1/2} \, \left|\psi\left(\frac{t-b}{a}\right)\right|\, dt\nonumber
\end{eqnarray}
\begin{eqnarray}
&&\quad \leq \,\epsilon\sum_{k=1}^K\,\left[\,a^{3/2}\,|\phi_k'(b)|\,\int\,|u|\,|\psi(u)|\,du
\,+\,a^{5/2}\,\frac{1}{2}\,M''_k\,\int\,|u|^2\,|\psi(u)|\,du\,\right]\nonumber\\
&&\quad\quad\quad\quad\,+\, \sum_{k=1}^K\,|A_k(b)|\,\epsilon\,\int\,\left[\,\frac{1}{2}\,|t-b|^2\,|\phi'_k(b)|\,+\,\frac{1}{6}\,|t-b|^3\,M''_k \right]\, a^{-1/2} \, \left|\psi\left(\frac{t-b}{a}\right)\right|\, dt
\nonumber\\
&&\quad \leq \epsilon\,a^{3/2}\,\left\{ I_1 \sum_{k=1}^K\,|\phi_k'(b)|
\,+\, \frac{1}{2}\,I_2\,a\,\sum_{k=1}^K\,\left[\,M''_k\,+\, |A_k(b)|\,|\phi_k'(b)|\, \right]
\,+\, \frac{1}{6}\,I_3\,a^2\,\sum_{k=1}^K\,M''_k\,|A_k(b)|\,\right\} \nonumber
\end{eqnarray}
\end{proof}

The wavelet $\psi$ satisfies $\widehat{\psi}(\xi)\neq 0$ only for $1-\Delta <\xi<1+\Delta$; it follows that $|W_f(a,b)|\,\leq\, \epsilon\,a^{3/2}\,\Gamma_1$ whenever
$|\,a\,\phi'_k(b)\,-\,1\,| \,>\,\Delta$ for all $k\in\{1,\ldots,K\}$. 
On the other hand, we also have the following lemma:

\begin{lemma} For any pair $(a,b)$ under consideration, there can be at most one $k \in \{1,\ldots,K\}$ for which $|\,a\,\phi'_k(b)\,-\,1\,| \,<\,\Delta$.
\label{zone_lemma}
\end{lemma}
\begin{proof}
Suppose that $k,\,\ell \in \{1,\ldots,K\}$ both satisfy the condition, i.e. that
$|\,a\,\phi'_k(b)\,-\,1\,| \,<\,\Delta$ and $|\,a\,\phi'_{\ell}(b)\,-\,1\,| \,<\,\Delta$, with
$k \neq \ell$. For the sake of definiteness, assume $k > \ell$. Since 
$f \in \mathcal{A}_{\epsilon,d}$, we have
\[
\phi'_k(b)-\phi'_{\ell}(b)\, \geq\,\phi'_k(b)-\phi'_{k-1}(b)\, \geq\,\
d\,[\phi'_k(b)+\phi'_{k-1}(b)]\, \geq\,\ d\,[\phi'_k(b)+\phi'_{\ell}(b)]~.
\]
Combined with
\begin{equation*}
\begin{aligned}
  \phi'_k(b)-\phi'_{\ell}(b)\, &\leq \,a^{-1}\,[\,(1+\Delta)\,-\,(1-\Delta)\,] \,=\,2\,a^{-1}\,\Delta~, \\
  \phi'_k(b)+\phi'_{\ell}(b)\, &\geq \,a^{-1}\,[\,(1-\Delta)\,-\,(1-\Delta)\,] \,=\,2\,a^{-1}\,(1-\Delta)~,
\end{aligned}
\end{equation*}
this gives
\[
\,\Delta \,\geq \,d\,(1-\Delta)~,
\]
which contradicts the condition $\Delta\,<\,d/(1+d)$ from Theorem \ref{mainthm}.
\end{proof}

It follows that the $a,b$-plane contains $K$ non-touching ``zones'',
corresponding to $|\,a\,\phi'_k(b)\,-\,1\,| \,<\,\Delta$, $k \in \{1,\ldots,K\}$, separated
by a ``no-man's land'' where $|W_f(a,b)|$ is small. We shall assume (see below) that 
$\epsilon$ is sufficiently small, i.e., that for all $(a,b)$ under consideration, 
\begin{equation}
\epsilon\,<\,a^{-9/4}\,\Gamma_1^{-3/2}~,
\label{first_cond_eps}
\end{equation}
so that  $\epsilon\,a^{3/2}\,\Gamma_1\,<\,\epsilon^{1/3}\,=\,\weps$. 
The upper bound in the intermediate region between the $K$ special zones 
is then below the threshold
allowed $\weps$ for the computation of $\omega_f(a,b)$ used in $S_{f,\weps}$ (see
the formulation of Theorem \ref{mainthm}). 
It follows that we
will compute $\omega_f(a,b)$ only in the special zones themselves. We thus need to estimate $\partial_bW_f(a,b)$ in each of these zones.

\begin{est}
For $k \in \{1,\ldots,K\}$, and $(a,b) \in \RR_+ \times \RR$ such that $|\,a\,\phi'_k(b)\,-\,1\,| \,<\,\Delta$, we have
\[
\left|\,-i \,\partial_b\,W_f(a,b)\,-\,\sqrt{2 \pi}\,A_k(b)\,e^{i\phi_k(b)}\,\sqrt{a}\,
\phi'_k(b)\,\widehat{\psi}\left(a\,\phi'_k(b)\right)\,\right|\,\leq \,\epsilon\,a^{1/2}\,\Gamma_2 ~,
\]
where
\[
\Gamma_2\,:=\, I'_1 \sum_{k=1}^K\,|\phi_k'(b)|
\,+\, \frac{1}{2}\,I'_2\,a\,\sum_{k=1}^K\,\left[\,M''_k\,+\, |A_k(b)|\,|\phi_k'(b)|\, \right]
\,+\, \frac{1}{6}\,I'_3\,a^2\,\sum_{k=1}^K\,M''_k\,|A_k(b)|~,
\]
with $I'_n\,:=\,\int\,|u|^n \, |\psi'(u)|\,du~$.
\label{est4}
\end{est}

\begin{proof}
The proof follows the same lines as that for Estimate \ref{est2}. 
We have
\begin{equation}
\begin{aligned}
  \partial_b \, W_f(a,b)\,&= \,\partial_b \,\left(\,\sum_{\ell=1}^K\,
    \int \, A_\ell(t) \, e^{i\phi_\ell(t)}\,
    a^{-1/2} \, \psi\left(\frac{t-b}{a}\right)\, dt\,\right) \\
  &=\,-\,a^{-3/2} \,\sum_{\ell=1}^K\, \int \, A_\ell(t) \,
  e^{i\phi_\ell(t)}\,
  \,\psi'\left(\frac{t-b}{a}\right)\, dt\, \\
  &=\,-\, \sum_{\ell=1}^K\,A_\ell(b)\,\int\,e^{i[\phi_\ell(b)
    \,+\,\phi_\ell'(b)\,(t-b)\,+\,\int_0^{t-b}[\phi_\ell'(b+u)
    -\phi_\ell'(b)]du]}\,
  a^{-3/2} \, \psi'\left(\frac{t-b}{a}\right)\, dt \\
  &\quad\quad\quad\,-\,\sum_{\ell=1}^K\,[A_\ell(t)\,-\,A_\ell(b)]\,
  e^{i\phi_\ell(t)}\, a^{-3/2} \, \psi'\left(\frac{t-b}{a}\right)\,
  dt~.  \nonumber
\end{aligned}
\end{equation}
By Lemma \ref{zone_lemma}, only the term for $\ell =k$
survives in the sum for $(a,b)$ such that $|\,a\,\phi'_k(b)\,-\,1\,| \,<\,\Delta$,
and we obtain
\begin{equation*}
\begin{aligned}
  &\,\left|\partial_b\,W_f(a,b)\,-\,i \,\sqrt{2
      \pi}\,A_k(b)\,e^{i\phi_k(b)}\,\sqrt{a}\,
    \phi'_k(b)\widehat{\psi}\left(a\,\phi'_k(b)\right)\,\right| \\
  &\quad\,=\,\left|\partial_b\,W_f(a,b)\,-\,\sqrt{2 \pi}\,A_k(b)\,e^{i\phi_k(b)}\,\frac{1}{\sqrt{a}}\,\widehat{\psi'}\left(a\,\phi'_k(b)\right)\,\right|\,\\
  &\quad\leq \,\int\,\epsilon\,|t-b|\left(|\phi'_k(b)|+\frac{1}{2}\,M''\,|t-b|\right)\,a^{-3/2} \, \left|\psi'\left(\frac{t-b}{a}\right)\right|\, dt \\
  &\quad\quad\quad\quad\,+\, |A_k(b)|\,\int\,\left|e^{i \int_0^{t-b}[\phi'_k(b+u)-\phi'_k(b)]du}\,-\,1\right|\,a^{-3/2} \, \left|\psi'\left(\frac{t-b}{a}\right)\right|\, dt \\
  &\quad \leq
  \,\epsilon\,\left[\,a^{1/2}\,|\phi_k'(b)|\,\int\,|u|\,|\psi'(u)|\,du
    \,+\,a^{3/2}\,\frac{1}{2}\,M''_k\,\int\,|u|^2\,|\psi'(u)|\,du\,\right]\\
  &\quad\quad\quad\quad\,+\, |A_k(b)|\,\epsilon\,\int\,\left[\,\frac{1}{2}\,|t-b|^2\,|\phi'_k(b)|\,+\,\frac{1}{6}\,|t-b|^3\,M''_k \right]\, a^{-3/2} \, \left|\psi'\left(\frac{t-b}{a}\right)\right|\, dt \\
  &\quad \leq \epsilon\,a^{1/2}\,\left\{ I'_1 \,|\phi_k'(b)| \,+\,
    \frac{1}{2}\,I'_2\,a\,\left[\,M''_k\,+\, |A_k(b)|\,|\phi_k'(b)|\,
    \right] \,+\, \frac{1}{6}\,I'_3\,a^2\,M''_k\,|A_k(b)|\,\right\}
\end{aligned}
\end{equation*}
\end{proof}
Combining Estimates \ref{est2} and \ref{est4}, we find 
\begin{est}
Suppose that {\rm(\ref{first_cond_eps})} is satisfied. 
For $k \in \{1,\ldots,K\}$, and $(a,b) \in \RR_+ \times \RR$ such that both $|\,a\,\phi'_k(b)\,-\,1\,| \,<\,\Delta$ and $W_f(a,b) \geq \weps$ are satisfied, we have
\[
\left| \, \omega_f(a,b) \,-\, \phi'_k(b)\, \right| \,\leq\, \sqrt{a}\,\left(\,\Gamma_2\,+\,a\,\Gamma_1\,\phi'_k(b)\,\right)\, \epsilon^{2/3}~.
\]
\label{est5}
\end{est}
\begin{proof}
By definition, 
\[
\omega_f(a,b)\,=\,\frac{-i \partial_bW_f(a,b)}{W_f(a,b)}~.
\]
For convenience, let us, for this proof only, denote $\sqrt{2 \pi}\,A_k(b)\,e^{i\phi_k(b)}\,\sqrt{a}\,\widehat{\psi}\left(a\,\phi'_k(b)\right)$ by B.
For the $(a,b)$-pairs under consideration, we have then
\begin{equation*}
|\,-i \,\partial_bW_f(a,b)\,-\,
\phi'_k(b)\,B\,| \,\leq\, \epsilon \,a^{1/2}\, \Gamma_2
\quad \mbox{and} \quad
|W_f(a,b)-\,B \,|\, \leq \,\epsilon\, a^{3/2}\, \Gamma_1~.
\end{equation*}
Using $\weps\,=\,\epsilon^{1/3}$, it follows that 
\[
\omega_f(a,b)\,-\,\phi'_k(b)\,=\,\frac{-i\partial_bW_f(a,b)\,-\,\phi'_k(b)\,B}{W_f(a,b)}\,+\,\frac{[B\,-\,W_f(a,b)]\,\phi'_k(b)}{W_f(a,b)}~,
\]
so that
\[
\left|\,\omega_f(a,b)\,-\,\phi'_k(b)\,\right| \,\leq \,\frac{\epsilon\,a^{1/2}\,\Gamma_2
\,+\,\epsilon\, a^{3/2}\, \phi'_k(b)\,\Gamma_1}{W_f(a,b)}
\,\leq\, \sqrt{a}\,\left(\,\Gamma_2\,+\,a\,\Gamma_1\,\phi'_k(b)\,\right)\, \epsilon^{2/3}~.
\]
\end{proof}

If (see below) we impose an extra restriction on $\epsilon$, namely that, for all
$(a,b)$ under consideration, and all $k \in \{1,\ldots,K\}\,$,
\begin{equation}
\epsilon \,\leq \, a^{-3/2}\, \left[ \,\Gamma_2 \,+\, a\, \phi'_k(b)\,\Gamma_1 \,\right]^{-3}~,
\label{second_cond_eps}
\end{equation}
then this last estimate can be simplified to
\begin{equation}
\left|\,\omega_f(a,b)\,-\,\phi'_k(b)\,\right| \,\leq \, \weps~.
\label{omega_ineq}
\end{equation}
Next is our final estimate:

\begin{est}
Suppose that both {\rm(\ref{first_cond_eps})} and {\rm(\ref{second_cond_eps})}
are satisfied, and that, in addition, for all $b$ under consideration,
\begin{equation}
\epsilon \,\leq\,1/8 \, d^3 \, [\, \phi'_1(b)\,+\,\phi'_2(b)\,]^3~.
\label{third_cond_eps}
\end{equation} 
Let $S_{f,\weps}$ be the synchrosqueezed wavelet transform of $f$,
\[
S_{f,\weps}(b,\omega)\,:=\,\int_{A_{\weps,f}(b)}\, W_f(a,b) \, 
\delta(\omega - \omega_f(a,b))\,a^{-3/2}\,da ~.
\]
Then we have, for all $b \in \RR$, and all $k \in \{1,\ldots,K\}$
\[
\left|\,\mathcal{R}_{\psi}^{-1} \int_{|\omega-\phi'_k(b)|<\weps}\,S_{f,\sigma}(b,\omega)\,d\omega
\,-\,A_k(b)\,e^{i \phi_k(b)}\,\right|\,\leq\,  C \weps~.
\]\label{est6}
\end{est}
\begin{proof}
For later use, note first that (\ref{third_cond_eps}) implies that,
for all $k,\,\ell \in \{1,\ldots,K\}$,
\begin{equation}
d \, [\, \phi'_k(b)\,+\,\phi'_{\ell}(b)\,] \,>\, 2 \, \weps~.
\label{auxiliary}
\end{equation}
We have 
\begin{equation*}
\begin{aligned}
  \int_{|\omega-\phi'_k(b)|<\weps}\,S_{f,\sigma}(b,\omega)\,d\omega
  \,&=\, \int_{|\omega-\phi'_k(b)|<\weps}\, \int_{A_{\weps, f}(b)} \,
  W_f(a,b)\,\delta(\omega\,-\,\omega_f(a,b))\,a^{-3/2}\, da   \,d\omega \\
  \,&=\,\int_{A_{\weps, f}(b) \cap \{ |\omega_f(a,b) - \phi_k'(b)
    |<\weps \}} \, W_f(a,b)\, a^{-3/2} \,da ~.
\end{aligned}
\end{equation*}
From Estimate \ref{est2} and (\ref{first_cond_eps}) we know that
$|W_f(a,b)|> \weps$ only when $|a\phi_\ell'(b) - 1 | < \Delta$ for
some $\ell \in \{1,\ldots,K\}$.  For $\ell \neq k$, we have (use
(\ref{auxiliary}))
\begin{equation*}
\begin{aligned}
|\omega_f(a,b)-\phi'_{\ell}(b)| \,&\geq\,|\phi'_{\ell}(b)-\phi'_k(b)|\,-\, 
|\omega_f(a,b)-\phi'_k(b)| \\
&\geq\, d\, [\,\phi'_{\ell}(b)+\phi'_k(b) \,]\,-\,\weps\,>\,\weps~,
\end{aligned}
\end{equation*} 
which, by Estimate \ref{est5}, implies that $|a\phi_{\ell}'(b) -
1|\geq\Delta$. Hence
\begin{equation*}
\begin{aligned}
  \int_{|\omega-\phi'_k(b)|<\weps}\,S_{f,\sigma}(b,\omega)\,d\omega
  \,&=\,\int_{A_{\weps,f}(b) \cap \{|a\phi_k'(b) -1|<\Delta \}}
  \,  W_f(a,t)\, a^{-3/2} \,da\\
  &=\, \left(\int_{|a\phi_k'(b) - 1|<\Delta} \,W_f(a,b)\, a^{-3/2}
    \,da \right)\, \\
  &\quad\quad\quad\quad-\, \left( \int_{ \{|a\phi_k'(b) -
      1|<\Delta\}\backslash A_{\weps,f}(b)}\,W_f(a,b)\, a^{-3/2}
    \,da\,\right)~.\nonumber
\end{aligned}
\end{equation*}
From Estimate \ref{est2} we then obtain
\begin{equation*}
\begin{aligned}
  &\left|\,\mathcal{R}_{\psi}^{-1}\,\int_{|\omega-\phi'_k(b)|<\weps}\,
    S_{f,\sigma}(b,\omega)\,d\omega \,-\,\,A_k(b)\,e^{i \phi_k(b)}\,\right|\,\\
  &\quad\quad\leq
  \,\left|\,\mathcal{R}_{\psi}^{-1}\,\left(\int_{|a\phi_k'(b)
        -1|<\Delta} \,W_f(a,b)\, a^{-3/2} \,da \right)\,-\,
    A_k(b)\,e^{i \phi_k(b)}\, \right| \, \\
  &\quad\quad\quad\quad\quad\quad +\,\mathcal{R}_{\psi}^{-1}\,
  \left|\int_{ \{|a\phi_k'(b) - 1|<\Delta\}\backslash A_{\weps,f}(b)}
    \,W_f(a,b)\, a^{-3/2} \,da \right| \\
  &\quad\quad\leq\,\left|\,\mathcal{R}_{\psi}^{-1}\,\sqrt{2 \pi}
    \,A_k(b)\,e^{i\phi_k(b)}\,\left(\int_{|a \phi'_k(b)-1|<\Delta)} \,
      \sqrt{a}\,\widehat{\psi}(a \phi'_k(b))\,a^{-3/2}
      \,da\,\right)\,-\, A_k(b)\,e^{i \phi_k(b)}\,\right|\\
  &\quad\quad\quad\quad\quad\quad +\,\mathcal{R}_{\psi}^{-1}\,\int_{|a
    \phi'_k(b)-1|<\Delta)}\,\left[\,\weps \,+\,\weps
    \,a^{-3/2}\,\right]\,da~.
\end{aligned}
\end{equation*}
For the first term on the right hand side, since
\begin{multline*}
\mathcal{R}_{\psi}^{-1}\,\sqrt{2 \pi}\,A_k(b)\,e^{i\phi_k(b)}\,\int_{|a \phi'_k(b)-1|<\Delta} \,\widehat{\psi}(a\phi'_k(b))\, a^{-1}\,da \\
=\,\mathcal{R}_{\psi}^{-1}\,\sqrt{2 \pi}\,A_k(b)\,e^{i\phi_k(b)}\,\int_{|\zeta -1|<\Delta}\,\widehat{\psi}(\zeta)\, \zeta^{-1}\,d\zeta\ \,=\,A_k(b)\,e^{i\phi_k(b)}~,
\end{multline*}
by the definition of $\mathcal{R}_{\psi}$, and hence the first term
vanishes. We thus obtain
\[ 
\left|\,\mathcal{R}_{\psi}^{-1}\,\int_{|\omega-\phi'_k(b)|<\weps}\,S_{f,\sigma}(t,\omega)\,d\omega \,-\,\,A_k(b)\,e^{i \phi_k(b)}\,\right|\,\leq \,2\,\weps\,\mathcal{R}_{\psi}^{-1}\,\left[\,\frac{\Delta}{\phi'_k(b)}
\,+\,  \left(\frac{\phi'_k(b)}{1-\Delta}\right)^{1/2} \,-\, \left(\frac{\phi'_k(b)}{1+\Delta}\right)^{1/2} \right]
\]
\end{proof}

It is now easy to see that all the Estimates together provide a
complete proof for Theorem \ref{mainthm}.

\begin{remark}
  We have three different conditions on $\epsilon$, namely
  (\ref{first_cond_eps}), (\ref{second_cond_eps}) and
  (\ref{third_cond_eps}). The auxiliary quantities in these
  inequalities depend on $a$ and $b$, and the conditions should be
  satisfied for all $(a,b)$-pairs under consideration. This is not
  really a problem: the dependence on $b$ is via the quantities
  $A_k(b)$ and $\phi'_k(b)$, and it is reasonable to assume these are
  uniformly bounded above and below (away from zero); because the
  bounds on the $\phi'_k(b)$ translate into bounds on $a$, we can
  likewise safely assume that $a$ is bounded above as well as below
  (away from zero). Note that the different terms can be traded off in
  many other ways than what is done here; no effort has been made to
  optimize the bounds, and they can surely be improved. The focus here
  was not on optimizing the constants, but on proving that, if the
  rate of change (in time) of the $A_k(b)$ and the $\phi'_k(b)$ is
  small, compared with the rate of change of the $\phi_k(b)$
  themselves, then synchrosqueezing will identify both the
  ``instantaneous
  frequencies'' and their amplitudes.\\
  In the statement of Theorem \ref{mainthm}, we required the wavelet
  $\psi$ to have a compactly supported Fourier transform. This is not
  absolutely necessary; it was done here for convenience in the proof.
  If $\widehat{\psi}$ is not compactly supported, then extra terms
  occur in many of the estimates, taking into account the decay of
  $\widehat{\psi}(\zeta)$ as $\zeta \rightarrow \infty$ or $\zeta
  \rightarrow 0$; these can be handled in ways similar to what we saw
  above, at the cost of significantly lengthening the computations
  without making a conceptual difference.
\end{remark}

\section{A Variational Approach}
\label{variational}
The construction and estimates in the previous section can also be
interpreted in a variational framework.

Let us go back to the notion of ``instantaneous frequency.'' 
Consider a signal $s(t)$ that is a
sum of IMT components $s_i(t)$:
\begin{equation}
  s(t) = \sum_{i=1}^N s_i(t)= \sum_{i=1}^N A_i(t) \,\cos(\phi_i(t)),
\end{equation}
with the additional constraints that $\phi_i'(t)$ and $\phi_j'(t)$ for
$i\not= j$ are ``well separated'', so that it is reasonable to
consider the $s_i$ as individual components. According to the
philosophy of EMD, the instantaneous frequency at time $t$, for the
$i$-th component, is then given by $\omega_i(t) = \phi_i'(t)$.

How could we use this to build a time-frequency representation for
$s$?  If we restrict ourselves to a small window in time around $T$,
of the type $[T-\Delta t,T+\Delta t]$, with $\Delta t\approx 2 \pi
/\phi_i'(T)$, then (by its IMT nature) the $i$-th component can be
written (approximately) as
\[
\left.s_i(t)_{\!_{\,}}\right|_{[T-\Delta t,T+\Delta t]}\approx A_i(T)\, \cos\left[\phi_i(T)+
\phi_i'(T)\,(t-T) \right],
\]
which is essentially a truncated Taylor expansion in which terms of size $O(A'_i(T))$, $O(\phi''_i(T))$
have been neglected. 
Introducing $\omega_i(T)= \phi_i'(T)$, and the phase 
$\varphi_i(T):=\phi_i(T)-\omega_i(T)T $
we can rewrite this as 
\begin{equation*}
\left.s_i(t)_{\!_{\,}}\right|_{[T-\Delta t,T+\Delta t]}\approx A_i(T)\, \cos\left[\omega_i(T)t+\varphi_i(T) \right].
\end{equation*}  

This signal has a time-frequency representation, as a bivariate
``function'' of time and frequency, given by (for $t \in [T-\Delta
t,T+\Delta t]$)
\begin{equation}
  F_i(t, \omega) =  A_i(T)\,\cos[\omega t +\varphi_i(T)]\,\delta(\omega - \omega_i(T)),
\end{equation}
where $\delta$ is the Dirac-delta measure.  The time-frequency
representation for the full signal $s=\sum_{i=1}^N s_i$, still in the
neighborhood of $t=T$, would then be
\begin{equation}
F(t, \omega) = \sum_{i=1}^N A_i(T)\,\cos[\omega t +\varphi_i(T)]\,\delta(\omega - \omega_i(T)).
\end{equation}
Integrating over $\omega$, in the neighborhood of $t=T$, leads to
$s(t) \approx \int F(t, \omega) \ud \omega$.

All this becomes even simpler if we introduce the ``complex form'' of
the time-frequency representation: for $t$ near $T$, we have
$\widetilde{F}(t,\omega)= \sum_{j=1}^N
\widetilde{A}_j(T)\,\exp(i\omega t)\,\delta(\omega-\phi'_j(T))$, with
$\widetilde{A}_j(T)=A_j(T)\,\exp[i\varphi_j(T)]$; integration over
$\omega$ now leads to
\begin{equation}
\label{eq:complex_int}
\mathfrak{Re}\left[ \int \widetilde{F}(t, \omega) \ud \omega \right] = 
\mathfrak{Re}\left[ \sum_{j=1}^N  A_j(T)\,\exp[i\omega_j(T)t+ i\varphi_j(T)]\,  \right] \approx
s(t).
\end{equation}
Note that, because of the presence of the $\delta$-measure, and under
the assumption that the components remain separated, the
time-frequency function $F(t,\omega)$ satisfies the equation
\begin{equation}
\label{eq:partial_id}
\partial_t \widetilde{F}(t,\omega)=i\omega \widetilde{F}(t,\omega).
\end{equation}

To get a representation over a longer time-interval, the small pieces
described above have to be knitted together. One way of doing this is
to set $\widetilde{F}(t,\omega)= \sum_{j=1}^N
\widetilde{A}_j(t)\,\exp(i\omega t)\,\delta(\omega-\omega_j(t))$.
This more globally defined $\widetilde{F}(t,\omega)$ is still
supported on the $N$ curves given by $\omega=\omega_j(t)$,
corresponding to the instantaneous frequency ``profile'' of the
different components. The complex ``amplitudes'' $\widetilde{A}_j(t)$
are given by $\widetilde{A}_j(t)={A}_j(t) \exp[i\varphi_j(t)]$, where,
to determine the phases $\exp[i\varphi_j(t)]$, it suffices to know
them at one time $t_0$. We have indeed
\[
\frac{\ud \varphi_j(t)}{\ud t} = \frac{\ud}{\ud t}\, [\phi_j(t)-\omega_j(t)t  ]
= \phi'_j(t)-\omega_j(t)-\omega'_j(t)t=-\omega'_j(t)t;
\]
since the $\omega_j(t)$ are known (they are encoded in the support of $\widetilde{F}$),
we can compute the $\varphi_j(t)$ by using 
\[\varphi_j(t)=\int_{t_0}^t\omega'_j(\tau)\tau 
\ud \tau \,+\, \varphi_j(t_0).
\]
Moreover, (\ref{eq:partial_id}) still holds (in the sense of distributions) 
up to terms of size 
$\mbox{O} (A'_i(T)), \, \mbox{O} (\phi''_i(T))$, since 
\begin{eqnarray*}
  \partial_t \widetilde{F}(t,\omega)
  &=&\sum_{j=1}^N 
  \left\{ \left[ A'_j(t) - i \omega'_j(t)t A_j(t)+i\omega A_j(t)\right]\right.
  \left.\,e^{i\omega t}\,\delta(\omega-\omega_j(t)) + A_j(t)\,e^{i\omega t}\,\omega'_j(t)\,\delta'(\omega-\omega_j(t))\right\}\\
  &=& i \omega \sum_{j=1}^N A_j(t)\,e^{i\omega t}\,\delta(\omega-\omega_j(t))\,+\, \mbox{O} (A'_i(T),\,\phi''_i(T))\\
  &=& i \omega \, \widetilde{F}(t,\omega)\,+\, \mbox{O} (A'_i(T), \,\phi''_i(T)).
\end{eqnarray*}

%
%

This suggests modeling the adaptive time-frequency decomposition as a
variational problem in which one seeks to minimize
  \begin{equation}\label{functional1}
    \int \left| \mathfrak{Re}\left[\int F(t,\omega) d\omega\right]-s(t)\right|^2 \ud t+\mu\iint \left|\partial_t F(t,\omega)-i\omega F(t,\omega)\right|^2 \ud t \ud \omega
  \end{equation}
  to which extra terms could be added, such as, $\gamma \iint
  |F(t,\omega)|^2 \ud t \ud \omega$ (corresponding to the constraint
  that $F \in L^2(\mathbb{R}^2)$), or $\lambda \int \left[ [\int
    |F(t,\omega)| \ud \omega\,\right]^2 \ud t$ (corresponding to a
  sparsity constraint in $\omega$ for each value of $t$). Using
  estimates similar to those in Section \ref{main_result}, one can
  prove that if $s \in \mathcal{A}_{\epsilon, d}$, then its
  synchrosqueezed wavelet transform $S_{s,\weps}(b,\omega)$ is close
  to the minimizer of (\ref{functional1}). Because the estimates and
  techniques of proof are essentially the same as in Section
  \ref{main_result}, we don't give the details of this analysis here.

  Note that wavelets or wavelet transforms play no role in the
  variational functional -- this fits with our numerical observation
  that although the wavelet transform itself of $s$ is definitely
  influenced by the choice of $\psi$, the dependence on $\psi$ is
  (almost) completely removed when one considers the synchrosqueezed
  wavelet transform, at least for signals in $\mathcal{A}_{\epsilon,
    d}$.

\section{Numerical Results}
\label{numerical}
In this section we illustrate the effectiveness of synchrosqueezed
wavelet transforms on several examples.  For all the examples in this
Section, synchrosqueezing was carried out starting from a Morlet
Wavelet transform; other wavelets that are well localized in frequency
give similar results.

\subsection{Instantaneous Frequency Profiles for Synthesized data}

We start by revisiting the toy signal of Figures
\ref{fig_linear_tf_methods} and \ref{fig_quadratic_tf_methods} in the
Introduction. Figure \ref{toy_synchr_squ} shows the result of
synchrosqueezing the wavelet transform of this toy signal.
\begin{figure}[ht]
\begin{center}
\begin{minipage}{.3\textwidth}
\begin{center}
\includegraphics[width=.9 \textwidth,height=1.3 in]{signal.png}
\end{center}
\end{minipage}
\hspace*{.03\textwidth}
\begin{minipage}{.3\textwidth}
\begin{center}
\includegraphics[ width=.9 \textwidth,height=1.3 in]{instfreq.png}
\end{center}
\end{minipage}
\hspace*{.03\textwidth}
\begin{minipage}{.3\textwidth}
\begin{center}
\includegraphics[width=.9 \textwidth,height=1.3 in]{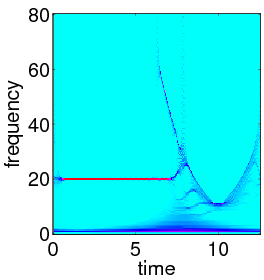}
\end{center}
\end{minipage}
\end{center}
\caption{\label{toy_synchr_squ}{\bf Revisiting the toy example from
    the Introduction.}  Left: the toy signal used for Figures
  \ref{fig_linear_tf_methods} and \ref{fig_quadratic_tf_methods};
  Middle: its instantaneous frequency; Right: the result of
  synchrosqueezing for this signal. The ``extra'' component at very
  low frequency is due to the signal's not being centered around 0.  }
\end{figure}
We next explore the tolerance to noise of synchrosqueezed wavelet
transforms.  We denote by $X(t)$ a white noise with zero mean and
variance $\sigma^2=1$. The {\em Signal-to-Noise Ratio (SNR)} (measured
in dB), will be defined (as usual) by
\begin{equation*}
  \text{SNR [dB]} = 10 \log_{10} \left( \frac{\operatorname{Var}f}{\sigma^2} 
  \right ),
\end{equation*}
where $f$ is the noiseless signal.
\begin{figure}[ht]
\centering
\begin{minipage}{0.25\textwidth}
\begin{center}
  \includegraphics[width=0.9\textwidth, height=1.1
  in]{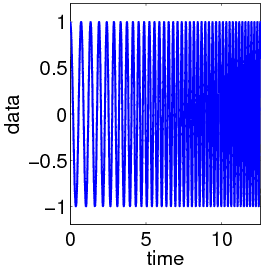}
\end{center}
\end{minipage}
\begin{minipage}{0.25\textwidth}
\begin{center}
  \includegraphics[width=0.9\textwidth, height=1.1
  in]{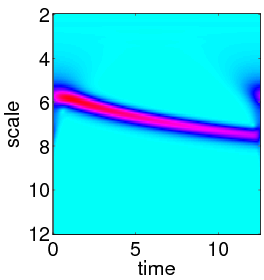}
\end{center}
\end{minipage}
\begin{minipage}{0.25\textwidth}
\begin{center}
  \includegraphics[width=0.9\textwidth, height=1.1
  in]{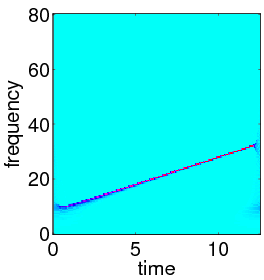}
\end{center}
\end{minipage}
  \label{fig:singlechirp:srwt-nonoise}
\begin{minipage}{0.25\textwidth}
\begin{center}
  \includegraphics[width=0.9\textwidth, height=1.1
  in]{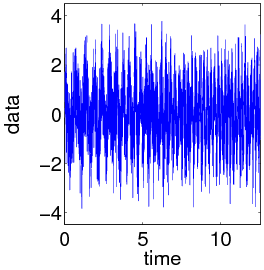}
\end{center}
\end{minipage}
\begin{minipage}{0.25\textwidth}
\begin{center}
  \includegraphics[width=0.9\textwidth, height=1.1
  in]{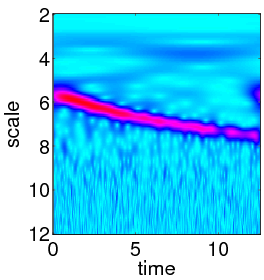}
\end{center}
\end{minipage}
\begin{minipage}{0.25\textwidth}
\begin{center}
  \includegraphics[width=0.9\textwidth, height=1.1
  in]{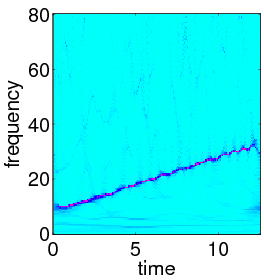}
\end{center}
\end{minipage}
\begin{minipage}{0.25\textwidth}
\begin{center}
  \includegraphics[width=0.9\textwidth, height=1.1
  in]{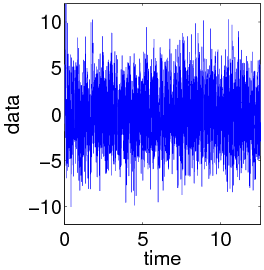}
\end{center}
\end{minipage}
\begin{minipage}{0.25\textwidth}
\begin{center}
  \includegraphics[width=0.9\textwidth, height=1.1
  in]{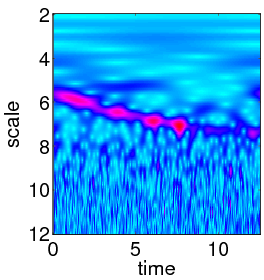}
\end{center}
\end{minipage}
\begin{minipage}{0.25\textwidth}
\begin{center}
  \includegraphics[width=0.9\textwidth, height=1.1
  in]{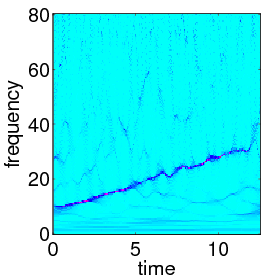}
\end{center}
\end{minipage}
\caption{\label{fig:singlechirp}Top row: left: single chirp signal
  without noise; middle: its continuous wavelet transform, and right:
  the synchrosqueezed transform.  Middle row: same, after white noise
  with SNR of $-3.00$ dB was added to the chirp signal.Lower row: same,
  now with white noise with SNR of $-12.55$ dB.}
\vspace*{.1 in}

%
\centering
\begin{minipage}{0.19\textwidth}
\hspace*{-0.004 \textwidth}
\includegraphics[width=0.99\textwidth, height=1.1 in]{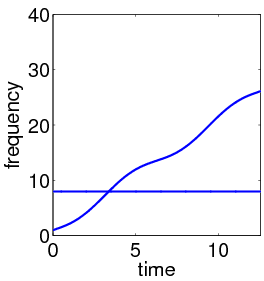}
\end{minipage}
\hspace*{0.002 \textwidth}
\begin{minipage}{0.19\textwidth}
\includegraphics[width=0.99\textwidth, height=1.07 in]{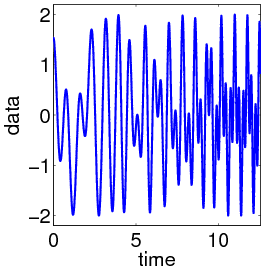}
\end{minipage}
\hspace*{-0.004 \textwidth}
\begin{minipage}{0.19\textwidth}
\includegraphics[width=0.99\textwidth, height=1.1in]{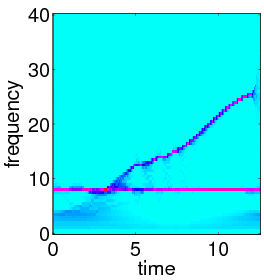}
\end{minipage}
\hspace*{0.002 \textwidth}
\begin{minipage}{0.19\textwidth}
\includegraphics[width=0.99\textwidth, height=1.1 in]{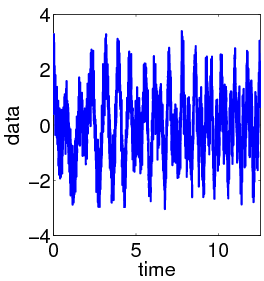}
\end{minipage}
\hspace*{-0.004 \textwidth}
\begin{minipage}{0.19\textwidth}
\includegraphics[width=0.99\textwidth, height=1.1 in]{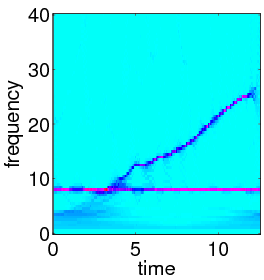}
\end{minipage}
\caption{\label{fig:crossover} {\em Far left:} the instantaneous
  frequencies, $\omega(t)=2t+1-\sin(t)$ and $8$, of the the two IMT components
  of the crossover signal $f(t) = \cos[t^2 + t +\cos(t)] + \cos(8t)$;
  {\em Middle left:} plot of $f(t)$ with no noise added;  {\em Middle:}
  Synchrosqueezed wavelet transforms of noiseless $f(t)$; {\em Middle right:}
  $f(t)+$noise (corresponding to SNR of $6.45$ dB); {\em Far right:} 
  synchrosqueezed wavelet transforms of 
  $f(t)+$noise.}
\end{figure}
Figure~\ref{fig:singlechirp} shows the results of applying our
algorithm to a signal consisting of one single chirp function $f(t) =
\cos(8t+t^2)$, without noise (i.e. the signal is just $f$), with some
noise (the signal is $f+X$, SNR$=-3.00$ dB), and with more
noise($f+3X$, SNR$=-12.55$dB).  Despite the high noise levels, the
synchrosqueezing algorithm identifies the component with reasonable
accuracy. Figure \ref{fig:singlechirp:instfreq} below shows the
instantaneous frequency curve extracted from these synchrosqueezed
transforms, for the three cases.

Finally we try out a ``crossover signal'', that is, a signal composed
of two components with instantaneous frequency trajectories that
intersect; in our example $f(t) = \cos(t^2 + t +\cos(t)) + \cos(8t)$.
Figure~\ref{fig:crossover} shows the signals $f(t)$ and
$f(t)+0.5X(t)$, together with their synchrosqueezed wavelet
transforms, as well as the ``ideal'' frequency profile given by the
instantaneous frequencies of the two components of $f$.
\newpage
\subsection{Extracting Individual Components from Synthesized data}
\begin{figure}[ht]
\begin{center}
\begin{minipage}{.19 \textwidth}
\begin{center}
  \includegraphics[width=0.9\textwidth, height=1.1
  in]{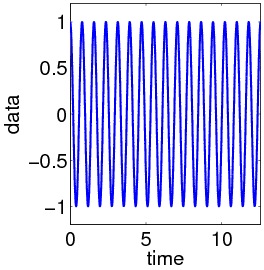}
\end{center}
\end{minipage}
\hspace*{-.002 \textwidth}
\begin{minipage}{.19 \textwidth}
\begin{center}
  \includegraphics[width=0.9\textwidth, height=1.1
  in]{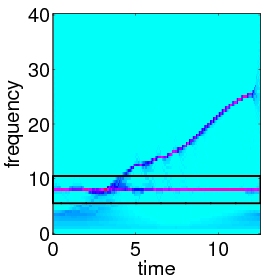}
\end{center}
\end{minipage}
\hspace*{-.002 \textwidth}
\begin{minipage}{.19 \textwidth}
\begin{center}
  \includegraphics[width=0.9\textwidth, height=1.1
  in]{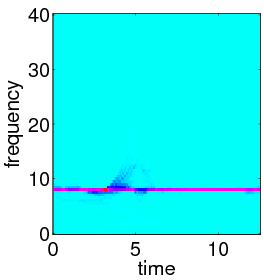}
\end{center}
\end{minipage}
\hspace*{-.002 \textwidth}
\begin{minipage}{.19 \textwidth}
\begin{center}
  \includegraphics[width=0.9\textwidth, height=1.1
  in]{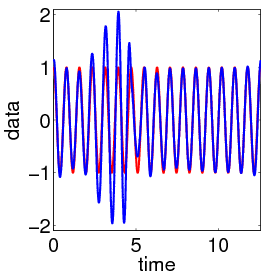}
\end{center}
\end{minipage}
\hspace*{-.002 \textwidth}
\begin{minipage}{.19 \textwidth}
\begin{center}
  \includegraphics[width=0.9\textwidth, height=1.1
  in]{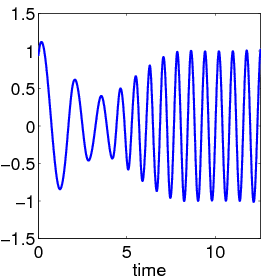}
\end{center}
\end{minipage}
%
\begin{minipage}{.19 \textwidth}
\begin{center}
  \includegraphics[width=0.9\textwidth, height=1.1
  in]{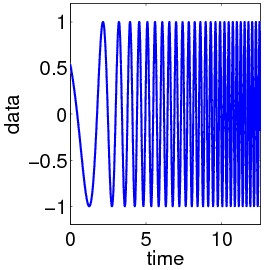}
\end{center}
\end{minipage}
\hspace*{-.002 \textwidth}
\begin{minipage}{.19 \textwidth}
\begin{center}
  \includegraphics[width=0.9\textwidth, height=1.1
  in]{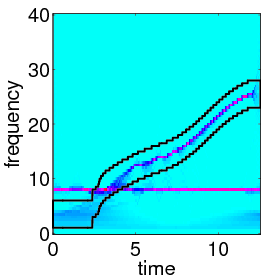}
\end{center}
\end{minipage}
\hspace*{-.002 \textwidth}
\begin{minipage}{.19 \textwidth}
\begin{center}
  \includegraphics[width=0.9\textwidth, height=1.1
  in]{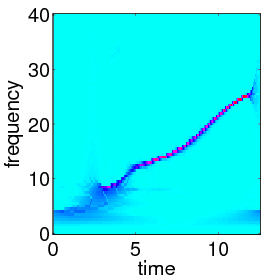}
\end{center}
\end{minipage}
\hspace*{-.002 \textwidth}
\begin{minipage}{.19 \textwidth}
\begin{center}
  \includegraphics[width=0.9\textwidth, height=1.1
  in]{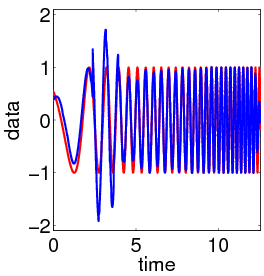}
\end{center}
\end{minipage}
\hspace*{-.002 \textwidth}
\begin{minipage}{.19 \textwidth}
\begin{center}
  \includegraphics[width=0.9\textwidth, height=1.1
  in]{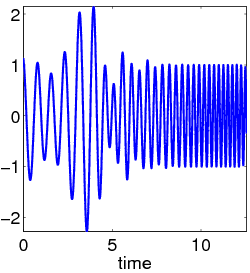}
\end{center}
\end{minipage}
%
\begin{minipage}{.19 \textwidth}
\begin{center}
  \includegraphics[width=0.9\textwidth, height=1.1
  in]{cross-imf1.png}
\end{center}
\end{minipage}
\hspace*{-.002 \textwidth}
\begin{minipage}{.19 \textwidth}
\begin{center}
  \includegraphics[width=0.9\textwidth, height=1.1
  in]{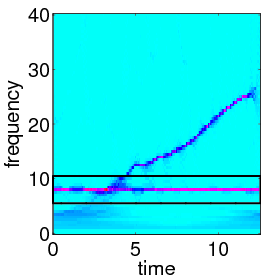}
\end{center}
\end{minipage}
\hspace*{-.002 \textwidth}
\begin{minipage}{.19 \textwidth}
\begin{center}
  \includegraphics[width=0.9\textwidth, height=1.1
  in]{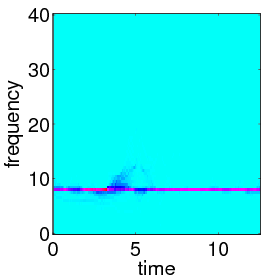}
\end{center}
\end{minipage}
\hspace*{-.002 \textwidth}
\begin{minipage}{.19 \textwidth}
\begin{center}
  \includegraphics[width=0.9\textwidth, height=1.1
  in]{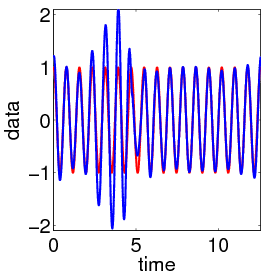}
\end{center}
\end{minipage}
\hspace*{-.002 \textwidth}
\begin{minipage}{.19 \textwidth}
\begin{center}
\includegraphics[width=0.9\textwidth, height=1.1 in]{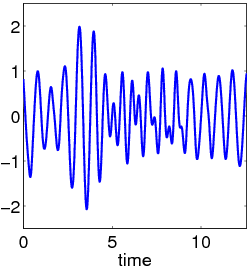}
\end{center}
\end{minipage}
\begin{minipage}{.19 \textwidth}
\begin{center}
  \includegraphics[width=0.9\textwidth, height=1.1
  in]{cross-imf2.png}
\end{center}
\end{minipage}
\hspace*{-.002 \textwidth}
\begin{minipage}{.19 \textwidth}
\begin{center}
  \includegraphics[width=0.9\textwidth, height=1.1
  in]{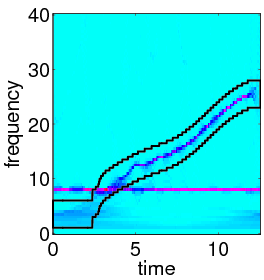}
\end{center}
\end{minipage}
\hspace*{-.002 \textwidth}
\begin{minipage}{.19 \textwidth}
\begin{center}
  \includegraphics[width=0.9\textwidth, height=1.1
  in]{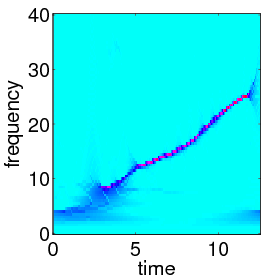}
\end{center}
\end{minipage}
\hspace*{-.002 \textwidth}
\begin{minipage}{.19 \textwidth}
\begin{center}
  \includegraphics[width=0.9\textwidth, height=1.1
  in]{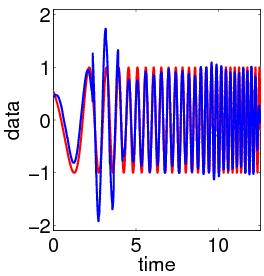}
\end{center}
\end{minipage}
\hspace*{-.002 \textwidth}
\begin{minipage}{.19 \textwidth}
\begin{center}
\includegraphics[width=0.9\textwidth, height=1.1 in]{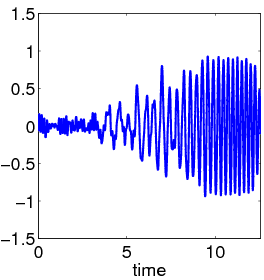}
\end{center}
\end{minipage}
\end{center}
\caption{
  \label{cross:comparison:imf} {\bf Comparing the decomposition into
    components $s_1(t)$ and $s_2(t)$ of the crossover signal
    $s(t)=s_1(t)+s_2(t)=\cos(8t)+\cos[t^2 + t +\cos(t)]$} {\bf Rows:}
  Noise-free situation in the first two rows; in the last two
  rows noise with SNR = $6.45$ dB was added to the mixed signal. In each
  case, $s_1$ is in the top row, and $s_2$ underneath. {\bf
    Columns:} Far left: true $s_j(t)$ $j=1,2$; Middle left: Zone marked on the
    synchrosqueezed transform for reconstruction of the component;
	Center: part of
  the synchrosqueezed transform singled out for the reconstruction of
  a putative $s_j$; Middle Right: the corresponding candidate $s_j(t)$
  according to the synchrosqueezed transform (plotted in blue over the
  original $s_j$, in red); Far right: candidate $s_j$
  according to EMD in the noiseless case, EEMD in the
  noisy case.  }
\end{figure}
In many applications listed in \cite{Costa:07, Cummings:04, Huang:98,
  HuangWu:08, WuHuang:09}, the desired end goal is the instantaneous
frequency trajectory or profile for the different components.  When
this is the case, the result of the synchrosqueezed wavelet transform,
as illustrated in the preceding subsection, provides a solution.

In other applications, however, one may wish to consider the
individual components themselves. These are obtained as an
intermediary result, before extracting instantaneous frequency
profiles, in the EMD and EEMD approaches. With synchrosqueezing, they
can be obtained in an additional step {\em after} the frequency
profiles have been determined.

Recall that, like most linear TF representations, the wavelet transform comes
equipped with reconstruction formulas,
\begin{eqnarray}
  f(t)&=& C_{\psi}\,\int_{-\infty}^{\infty} \int_0^{\infty} W_f(a,b) \,a^{-5/2} \,\psi\left(\frac{t-b}{a} \right)\,\ud a \ud b ~,\label{2_int_reconstr}\\
  \mbox{as well as }
  f(t)&=& C'_{\psi}\,\int_0^{\infty} W_f(a,t) \,a^{-3/2}\,\ud a
  = C'_{\psi}\,\mathfrak{Re}\left[\,\int_0^{\infty} T_f(t,\omega) \ud \omega\,\right],
\label{1_int_reconstr}
\end{eqnarray}
where $C_{\psi},\,C'_{\psi}$ are constants depending only on $\psi$.
For the signals of interest to us here, the synchrosqueezed
representation has, as illustrated in the preceding subsection, well
localized zones of concentration.  One can use these to select the
zone corresponding to one component, and then integrate, in the
reconstruction formulas, over the corresponding subregion of the
integration domain. In practice, it turns out that the best results
are obtained by using the reconstruction formula
(\ref{2_int_reconstr}).

We illustrate this with the crossover example from the previous
subsection:  Figure~\ref{cross:comparison:imf} shows the zones
selected in the synchrosqueezed transform plane as well as the 
corresponding reconstructed components.  This figure
also shows the components obtained for these signals by EMD for the
clean case, and by EEMD (more robust to noise than EMD) for the noisy
case; for this type of signal, the synchrosqueezed transform proposed here
seems to give a more easily interpretable result. 

Once the individual components are extracted, one can use them to get
a numerical estimate for the variation in time of the instantaneous
frequencies of the different components.  To illustrate this, we
revisit the chirp signal from the previous subsection.  Figure
\ref{fig:singlechirp:instfreq} shows the frequency curves obtained by
the synchrosqueezing approach; they are fairly robust with
respect to noise.  
\begin{figure}[ht]
\centering
\begin{minipage}{0.3\textwidth}
\begin{center}
  \includegraphics[width=0.8\textwidth, height=1.1
  in]{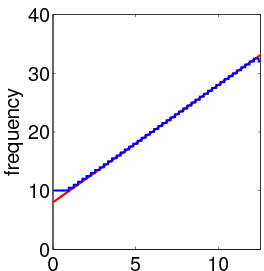}
\end{center}
\end{minipage}
\begin{minipage}{0.3\textwidth}
\begin{center}
  \includegraphics[width=0.8\textwidth, height=1.1
  in]{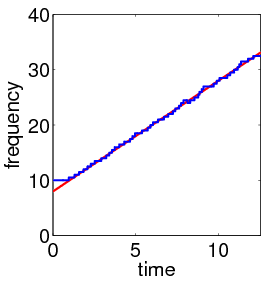}
\end{center}
\end{minipage}
\begin{minipage}{0.3\textwidth}
\begin{center}
  \includegraphics[width=0.8\textwidth, height=1.1
  in]{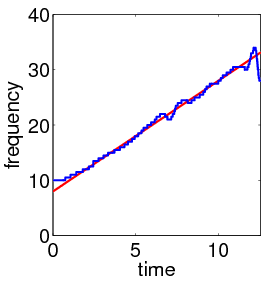}
\end{center}
\end{minipage}
\caption{\label{fig:singlechirp:instfreq} {\bf Instantaneous frequency
    curves extracted from the synchrosqueezed representation} Left:
  the instantaneous frequency of the clean single chirp signal
  estimated by synchrosqueezing; Middle: the instantaneous frequency
  of the fairly noisy single chirp signal (SNR=$-3.00$dB) estimated by
  synchrosqueezing. Right:the instantaneous frequency of the noisy
  single chirp signal (SNR=$-12.55$dB) estimated by synchrosqueezing.}
\end{figure}
\subsection{Applying the synchrosqueezed transform to some real data}

So far, all the examples shown concerned toy models or synthesized data.
In the subsection we illustrate the result on some real data sets, of medical origin.

{\bf 5.3.A $~$Surface Electromyography Data}

In this application, we ``clean up''
surface
electromyography (sEMG) data acquired from a healthy young female with
a portable system (QuickAmp). The sEMG electrodes, with sensors of
high-purity sintered Ag/AgCl, were placed on the triceps. The signal
was measured for 608 seconds, sampled at 500Hz. During the data
acquisition, the subject flexed/extended her elbow, according to a
protocol in which instructions to flex the right or left elbow were
given to the subject at not completely regular time intervals; the
subject did not know prior to each instruction which elbow she would
be told to flex, and the sequence of left/right choices was random.
The raw sEMG data $s_l(t)$ and $s_r(t)$ are shown in the left column in
Figure~\ref{signal}.

\begin{figure}[ht]
\begin{minipage}{.33\textwidth}	
\includegraphics[width=0.99\textwidth, height=1.5 in]{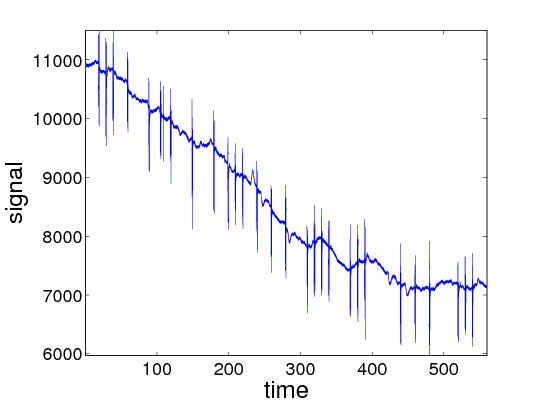}
$~~~~$
\end{minipage}
\hspace{.004\textwidth}
\begin{minipage}{.33\textwidth}	
\includegraphics[width=0.99\textwidth, height=1.5 in]{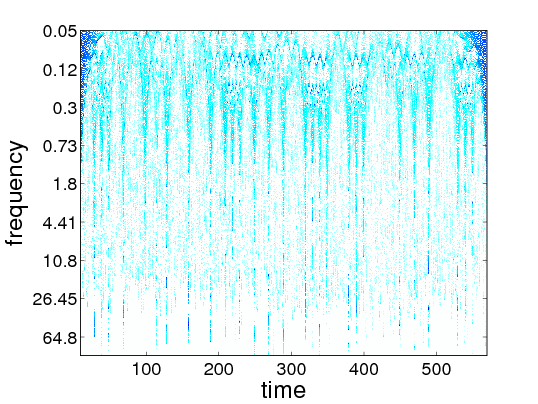}
\end{minipage}
\hspace{.004\textwidth}
\begin{minipage}{.33\textwidth}
\includegraphics[width=0.99\textwidth, height=1.5 in]{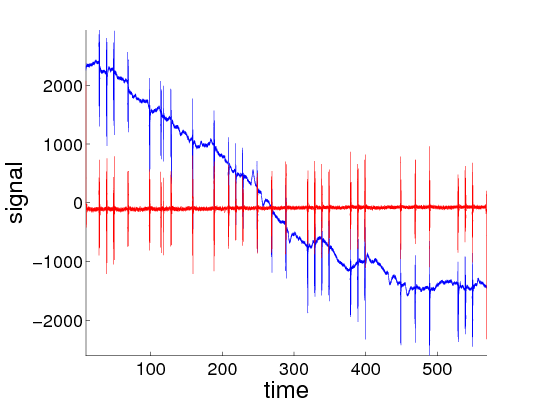}
$~$
\end{minipage}
\begin{minipage}{.33\textwidth}
\includegraphics[width=0.99\textwidth, height=1.5 in]{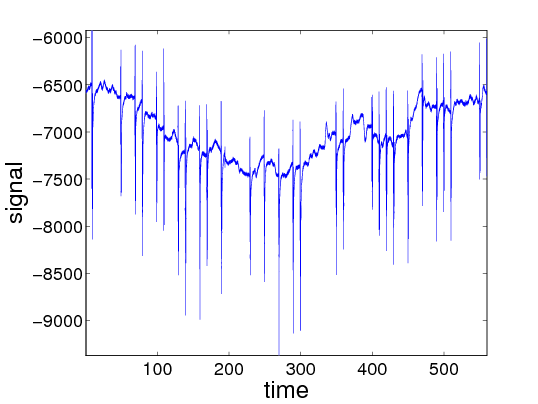}
$~$
\end{minipage}
\hspace{.004\textwidth}
\begin{minipage}{.33\textwidth}
\includegraphics[width=0.99\textwidth, height=1.5 in]{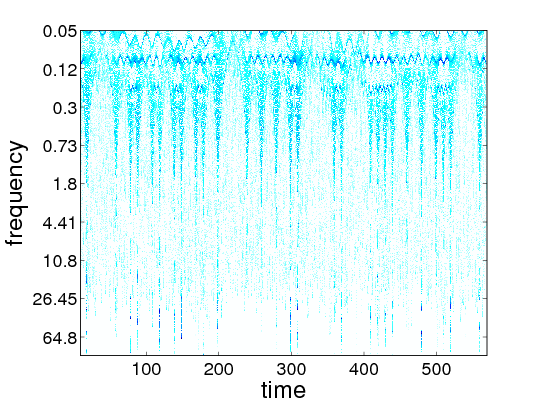}
$~$
\end{minipage}
\hspace{.004\textwidth}
\begin{minipage}{.33\textwidth}
\includegraphics[width=0.99\textwidth, height=1.5 in]{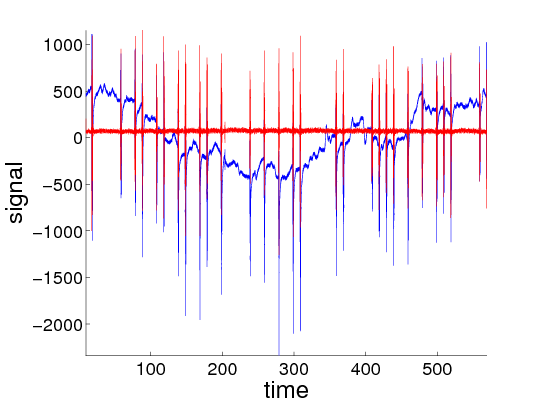}
$~$
\end{minipage}
\caption{{\em Left column:} The surface electromyography data described in the text. The erratic 
(and different) drifts present in these signals are pretty typical. The very short-lived
peaks correspond to the elbow-flexes by the subject. Note that the
peaks have, on average, the same amplitudes for the left and right triceps; they look
larger for the right triceps only because of a difference in scale.
{\em Middle column:} The synchrosqueezing transforms of the surface electromyography signals.
Coarse scales are near the top of these diagrams, finer scales are shown lower. The
peaks are clearly marked over a range of fine scales. {\em Right column:}
The red curve give the signals ``sans drift'', reconstructed by deleting the low frequency region in the synchrosqueezed transforms; comparison with the original sEMG signals (in blue) shows the R peaks are as sharp as in the original signals, and at precisely the
same location.
}\label{signal}
\end{figure}

The middle column in Figure~\ref{signal} 
shows the results $\mathcal{W}s_{\ell}(t,\omega)$,
$\mathcal{W}s_r(t,\omega)$ of our synchrosqueezing algorithm applied
to the two sEMG data sets. We used an implementation in which each
dyadic scale interval ($a\in[2^k,2^{k+1}]$) was divided into 32
equi-log-spaced bins.

The original surface electromyography signals show an erratic drift,
which medical researchers wish to remove without losing any sharpness
in the peaks. To achieve this, we identified the low frequency
components (i.e. the dominant components at frequencies below 1 Hz) in
the signal, and removed them before reconstruction. More, precisely,
we defined $\widetilde{s}_i(t)=\sum_{\xi\geq
  \xi_{i,\mbox{\tiny{cut-off}}}}\mathcal{T}s_{i}(t,\xi)$, with
frequency cut-off $\xi_{i,\mbox{\footnotesize{cut-off}}}(t)=
\omega_i(t)+\omega_0$, where $\omega_i(t)$ was the dominant component
for signal $i$ ($i=\ell$ or $r$) near 1 Hz with the highest frequency,
and $\omega_0$ a small constant offset.  The right column in Figure
\ref{signal} shows the results; the erratic drift has disappeared and
the peaks are well preserved. Note that a similar result can also be
obtained straightforwardly from a wavelet transform without any
synchrosqueezing.

{\bf 5.3.B. Electrocardiogram Data}

In this application we use synchrosqueezing to extract the heart rate
variability (HRV) from a real electrocardiogram (ECG) signal. The data
was acquired from a resting healthy male with a portable ECG machine
at sampling rate $1000$Hz for $600$ seconds. The samples were
quantized at $12$ bits across $\pm 10$ mV.  The raw ECG data $e(t)$
are (partially) shown in the left half of Figure~\ref{ecgsig}; the
right half of the figure shows a blow-up of 20 seconds of the same
signal.

\begin{figure}[ht]
\begin{minipage}{.65\textwidth}
\includegraphics[width=.99\textwidth, height= 2.2 in]{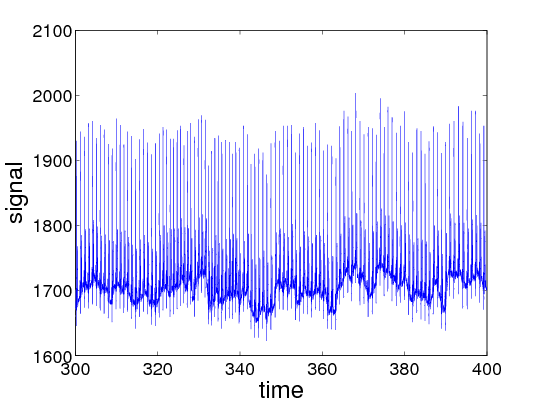}
\end{minipage}
\begin{minipage}{.34\textwidth}
\hspace*{-.1 \textwidth}
\includegraphics[width=1.2\textwidth, height= 2.2in]{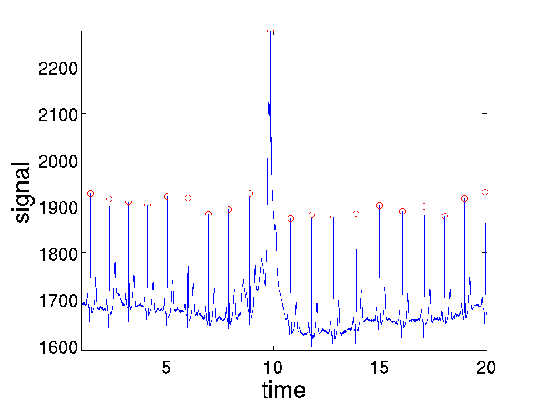}
\end{minipage}
\caption{{\bf The raw electrocardiogram data} {\em Left:} portion
  between 300 and 400 sec; {\em Right:} blow-up of the stretch between
  320 and 340 sec, with the R peaks marked.}\label{ecgsig}
\end{figure}

The strong, fairly regular spikes in the ECG are called the R peaks;
the heart rate variability (HRV) time series is defined as the
sequence of time differences between consecutive R peaks. (The
interval between two consecutive R peaks is also called the RR
interval.)  The HRV is important for both clinical and basic research;
it reflects the physiological dynamics and state of health of the
subject. (See, e.g., \cite{hrvbook} for clinical guidelines pertaining
to the HRV, and \cite{hrv2006} recent advances made in research.)  The
HRV can be viewed as a succession of snapshots of an averaged version
of the instantaneous heart rate.

The left half of Figure \ref{ecgsrwt} shows the synchrosqueezed
transform $T_e(\omega,t)$ of $e(t)$; in this case we used an
implementation in which each dyadic scale interval
($a\in[2^k,2^{k+1}]$) was divided into $128$ equi-log-spaced bins.
The synchrosqueezed transform $T_e(\omega,t)$ has a dominant line
$c(t)$ near $1.2$Hz, the support of which can be parametrized as
$\{(t,\omega_c(t)); \, t \in [0,80 \mbox{sec}]\,\}$.  The right half
of Figure \ref{ecgsrwt} tracks the dependence on $t$ of
$\omega_c(t)$. This figure also plot a (piecewise constant) function
$f(t)$ that tracks the HRV time series and that is computed as
follows: if $t$ lies between $t-i$ and $t_{i+1}$, the locations in
time for the $i$-th and $(i+1)$-st R peaks, then
$f(t)=[t_{i+1}-t_i]^{-1}$. The plot for $\omega(t)$ and $f(t)$ are
clearly highly correlated.

\begin{figure}[ht]
\begin{minipage}{.49 \textwidth}
  \includegraphics[width=0.99\textwidth, height =1.5
  in]{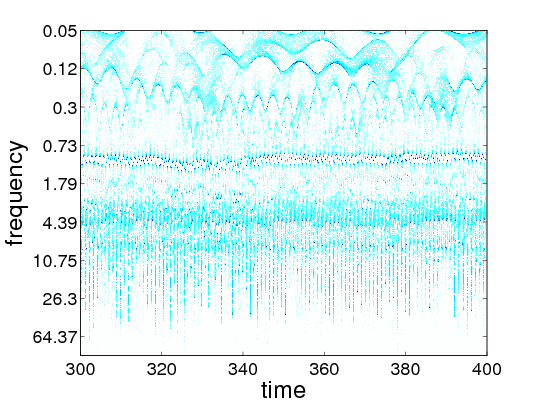}
\end{minipage}
\hspace*{.01 \textwidth}
\begin{minipage}{.49 \textwidth}
\includegraphics[width=.99\textwidth, height= 1.5 in]{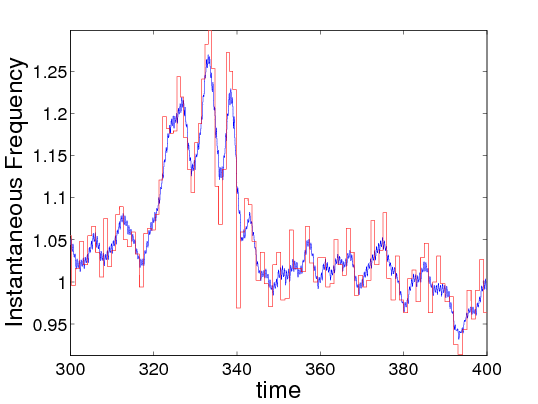}
\end{minipage}
\caption{{\em Left:} The synchrosqueezing transforms of the electrocardiogram signals
given in Fig. \ref{ecgsig}. {\em Right:} The blue curve shows the ``instantaneous heart rate'' $\omega(t)$ computed by tracking 
the support of the dominant curve in the synchrosqueezed transform $T_e$; the
red curve is the (piecewise constant) inverse of the successive $\mbox{RR}_i$
\label{ecgsrwt}}
\end{figure}

\section{Acknowledgments}
The authors are grateful to the Federal Highway Administration, which
supported this research via FHWA grant DTFH61-08-C-00028. They also
thank Prof. Norden Huang and Prof. Zhaohua Wu for many stimulating
discussions and their generosity in sharing their code and
insights. They also thank MD. Shu-Shya Hseu and Prof. Yu-Te Wu for
providing the real medical signal.  

\bibliographystyle{amsplain}
\bibliography{main}

\end{document}